\documentclass[12pt,reqno,oneside]{amsart}
\usepackage{amscd,amssymb,amsmath,amsthm}
\usepackage[dvips]{graphicx}
\usepackage{mathrsfs}

\newtheorem{thm}{Theorem}[section]
\newtheorem{lem}[thm]{Lemma}

\newtheorem{defn}[thm]{Definition}
\newtheorem{cor}[thm]{Corollary}
\newtheorem{rem}[thm]{\textit{\textrm{Remark}}}
\newtheorem{notn}[thm]{Notation}
\newtheorem{ex}[thm]{\textit{\textrm{Example}}}
\newtheorem{note}{\textit{\textrm{Note}}}
\numberwithin{equation}{section}

\usepackage{hyperref}

\numberwithin{equation}{section}
\newcommand{\NI}{\noindent}
\newcommand{\ds}{\displaystyle}

\newcommand\HUGE{\@setfontsize\Huge{38}{47}}
\begin{document}
\title[Cellularity of a larger class of diagram algebras] { Preprint \\ \vspace{1cm} Cellularity of a larger class of diagram algebras}

 \author{N. Karimilla Bi}
 \maketitle{\small{

\begin{center}
Ramanujan Institute for Advanced Study in Mathematics, \,\\
University  of  Madras,  \\
Chepauk, Chennai -600 005, Tamilnadu, India.\\
{\bf {E-Mail: karimilla.riasm@gmail.com}}
\end{center}}
\begin{abstract}
In this paper, we realize the algebra of $\mathbb{Z}_2$-relations, signed partition algebras and partition algebras as tabular algebras and prove the cellularity of these algebras using the method of \cite{GM1}.
 Using the results of Graham and Lehrer in \cite{GL}, we
give the modular representations of the algebra of $\mathbb{Z}_2$-relations, signed partition algebras and partition algebras.
\end{abstract}

\section{\textbf{INTRODUCTION}}

The study of the algebra of $\mathbb{Z}_2$-relations and signed partition algebras are important because as they are subalgebras of partition algebras which  arose naturally as potts model in statistical mechanics. In this paper, we establish the cellularity of the algebra of $\mathbb{Z}_2$-relations and signed partition algebras and hence deduce the modular representations of these algebras. The algebra of $\mathbb{Z}_2$-relations and signed partition algebras are different from the $\mathbb{Z}_2$-colored partition algebra introduced  in \cite{M} and Tanabe algebras introduced in \cite{T} which are explained in section 3.

\section{\textbf{Preliminaries}}

\begin{defn} \textbf{[\cite{VSS}]} \label{D2.1}

Let the group $\mathbb{Z}_2$ act on the set $X$. Then the action of $\mathbb{Z}_2$ on $X$
can be extended to an action of $\mathbb{Z}_2$ on $R(X), $ where $R(X)$
denote the set of all equivalence relations on $X,$ given by

\centerline{$g . d = \{ (gp, gq) \ | \ (p, q) \in d \}$}

\NI where $d \in R(X)$ and $g \in \mathbb{Z}_2.$ (It is easy to see that the
relation $g . d$ is again an equivalence relation).

An equivalence relation $d$ on $X$ is said to be a $\mathbb{Z}_2$-stable
equivalence relation  if $p \sim q$ in $d$ implies that $gp \sim
gq$ in $d$ for all $g$ in $\mathbb{Z}_2.$  We denote $[k]$ for the set $\{1,
2, \ldots, k\}.$ We shall only consider the case when $\mathbb{Z}_2$ acts
freely on $X$; Let $X  := [k] \times \mathbb{Z}_2 $
 and the action is defined by $g.(i, x) = (i, gx)$ for all $1
\leq i \leq k.$ Let $R_k^{\mathbb{Z}_2}$ be the set of all $\mathbb{Z}_2$-stable
equivalence relations on $X.$
\end{defn}

\begin{notn}\label{N2.2}
$R_k^{\mathbb{Z}_2}$ denotes the set of all $\mathbb{Z}_2$-stable
    equivalence relations on $ \{1,2,\cdots,k\} \times \mathbb{Z}_2.$

    Each element $d \in R_k^{\mathbb{Z}_2}$ can be represented as  a simple graph on a
    row of $2k$ vertices.

    \begin{enumerate}
        \item [(i)] The vertices $(1,e),(1,g),\cdots,(k,e),(k,g)$
        are arranged from left to right in a single row.
        \item [(ii)] If $(i,g) \sim (j,g') \in R_k^{\mathbb{Z}_2}$
        then $(i, g), (j, g')$ is joined by a line $\forall g, g' \in \mathbb{Z}_2.$
    \end{enumerate}

    We say that the two graphs are equivalent if they give rise to
    the same set partition of the $2k$ vertices $\{ (1,e), (1,g), \cdots, (k, e), (k, g)\}.$ We may regard each element $d$ in $R_{2k}^{\mathbb{Z}_2}$ as a
$2k$-partition diagram by arranging the $4k$ vertices $(i, g), i
\in [2k], g \in \mathbb{Z}_2$ of $d$ in two rows in  such a way
that $(i, g)$ is in the top(bottom) row of $d$ if $1 \leq i \leq k
\big(k + 1 \leq i \leq 2k \big) \ \ \forall \ g \in \mathbb{Z}_2$  and put
$(k + i, g) = (i', g), 1 \leq i \leq k,$ for all $g \in
\mathbb{Z}_2$ in the bottom row of $d$ and if $(i, g) \sim (j,
g')$ then $(i, g), (j, g')$ is joined by a line $\forall \ g, g'
\in \mathbb{Z}_2.$

The diagrams $d^+$ and $d^-$ are obtained from the diagram $d$ by
restricting the vertex set to $\{ (1, e), (1, g), \ldots, $

\NI $(k, e),
(k, g)\}$ and $\{ (1', e), (1', g), \ldots, (k', e), (k', g)\}$
respectively.The diagrams $d^+$ and $d^-$ are also
$\mathbb{Z}_2$-stable equivalence relation and  $d^+, d^- \in
R_k^{\mathbb{Z}_2}.$
\end{notn}

\begin{defn} \textbf{[\cite{VSS}]} \label{D2.3}

Let $d \in R_{2k}^{\mathbb{Z}_2}.$ Then the equation

\centerline{$R^d = \{ (i, j) \ | \ \text{ there exists } g, h \in \mathbb{Z}_2
\text{ such that } ((i, g), (j, h)) \in d \}$}

 defines an equivalence relation
on $[2k].$
\end{defn}

\begin{rem} \textbf{[\cite{VSS}]} \label{R2.4}
For $d \in R_{2k}^{\mathbb{Z}_2}$ and for every $\mathbb{Z}_2$-stable equivalence class or a connected component
$C$ in $R^d$ there exists a unique subgroup denoted by $H_C^d$
where
\begin{enumerate}
    \item [(i)] $H_C^d = \{ e \}$ if $(i, e) \nsim (i,
    g) \ \ \ \forall i \in C$ , $C$ is called an $\{e\}$-class or $\{e\}$-component  and the $\{e\}$ component $C$ will always occur as a pair and
    \item [(ii)] $H_C^d = \mathbb{Z}_2$ if  $(i, e) \sim (i, g) \ \ \ \forall i \in
    C,$ $C$ is called $\mathbb{Z}_2$-class or
    $\mathbb{Z}_2$-component and the number of vertices in the $\mathbb{Z}_2$-component $C$ will always be even.
\end{enumerate}
\end{rem}

\begin{defn} \textbf{[\cite{VSS}]} \label{D2.5}

The linear span of $R_{2k}^{\mathbb{Z}_2}$ is a subalgebra of
$\mathbb{P}_{2k }(x)$. We denote this subalgebra by
$A_k^{\mathbb{Z}_2}(x), $ called the \textbf{algebra of
$\mathbb{Z}_2$-relations}.
\end{defn}

\begin{defn} \textbf{[\cite{VSS}]} \label{D2.6}

Let $d$ be a $2k$-partition diagram. A connected component $C$ of
$d$ which contains vertices in both the rows, is called a through class
of $d$ and  $\sharp^p(d)$ denotes the number of through classes of
$d$, called propagating number. Any connected component $C$ of $d$ which contains vertices in only one row (either a top row or bottom row) is called a horizontal edge.

\NI For $0 \leq 2s_1+s_2 \leq 2k, $ define $I_{2s_1+s_2}^{2k}$ to be the set of all
$2k$-partition diagrams such that $\sharp^p(d) = 2s_1+s_2 $ for all $d \in I_{2s_1+s_2}^{2k}.$

\NI i.e., $d$ has $s_1$ number of pairs of $\{e\}$ through classes and $s_2$ number of $\mathbb{Z}_2$ through classes.

\NI Let $I_s$ be
the linear space spanned by $\ds \bigcup_{2s_1+s_2 \leq s} I_{2s_1+s_2}^{2k}.$

\end{defn}

\begin{defn} \textbf{(\cite{P}, Definition 3.1.1)}\label{D2.7}

Let the \textbf{signed partition algebra
$\overrightarrow{A}_k^{\mathbb{Z}_2}(x)$} be the subalgebra of
$\mathbb{P}_{2k}(x)$ generated by

\begin{center}
\vspace{-0.3cm}
\includegraphics{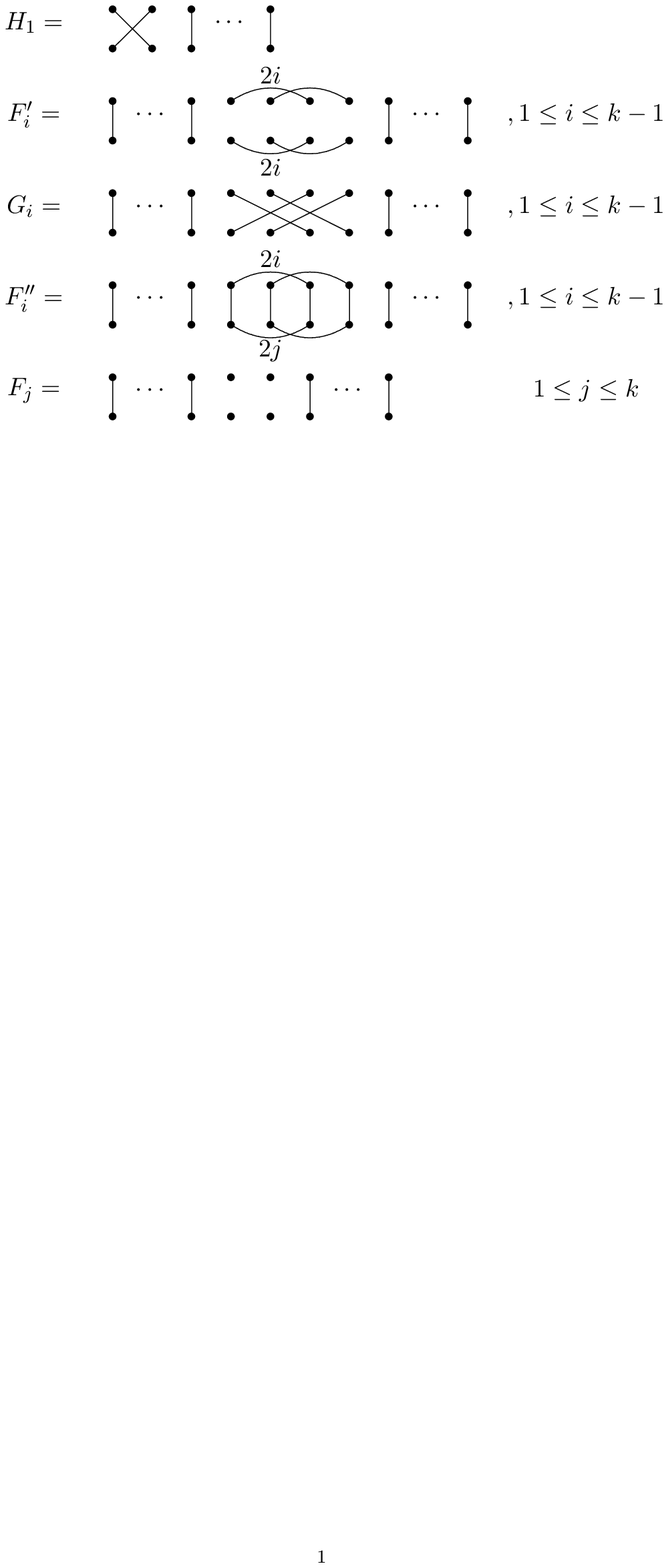}
\end{center}
The subalgebra of the signed partition algebra generated by $F_i', G_i, F''_i, F_j, 1 \leq i \leq k-1, 1 \leq j \leq k$ is isomorphic on to the partition algebra $\mathbb{P}_{2k}(x^2).$
\end{defn}

\begin{defn} \textbf{(\cite{P}, Definition 3.1.1)}\label{D2.8}

Let $d \in R_{2k}^{\mathbb{Z}_2}.$ For $0 \leq r =2s_1+s_2 \leq 2k
- 1, 0 \leq s_1, s_2 \leq k - 1,$

\NI $\widetilde{I}_{2s_1+s_2}^{2k} = \Big\{ d \in I_{2s_1+s_2}^{2k} \ | \ s_1 + s_2 +  H_e(d^{+}) + H_{\mathbb{Z}_2}
(d^{+}) \leq k - 1 \text{ and } \ s_1 + s_2  + H_e(d^{-}) +H_{\mathbb{Z}_2} (d^{-}) \leq k - 1 \Big\}$,

where
\begin{enumerate}
    \item [(i)] $s_1 = \natural \left\{ C : C \text{ is a through class of } R^d \text{ such that } H_C^d = \{ e \} \right\},$
    \item[(ii)] $s_2 = \natural \left\{ C : C \text{ is a through class of } R^d \text{ such that } H_C^d =  \mathbb{Z}_2  \right\},$
        \item[(iii)]$H_e(d^{+}) \left( H_e(d^{-})\right)$ is the number of $\{e\}$ horizontal edges $C$ in the top(bottom) row of $R^d$ such that $H_C^d = \{ e \}$ and $| C |
\geq 2,$
\item[(iv)] $H_{\mathbb{Z}_2}(d^{+}) \left(
H_{\mathbb{Z}_2}(d^{-})\right)$ is the number of $\mathbb{Z}_2$ horizontal edges $C$ in the top(bottom) row of $R^d$ such that $H_C^d =
\mathbb{Z}_2.$
\item[(v)] $\sharp^p (R^d) = s_1+s_2.$
\end{enumerate}
Put, $\widetilde{I}^{2k}_r = \underset{2s_1+s_2 \leq r}{\cup} \widetilde{I}^{2k}_{2s_1+s_2}.$
\end{defn}

\begin{defn} \textbf{(\cite{P}, Definition 3.1.1)} \label{D2.9}

When $s_1 = k, r = 2k, \widetilde{I}^{2k}_{r} = I^{2k}_{2k}.$

Let  $\ds \widetilde{I}_{2k} =  \underset{r=
0}{\overset{2k}\cup} \widetilde{I}_{r}^{2k} .$ The linear span of $\widetilde{I}_{2k}$
is denoted by $\mathscr{H}.$
\end{defn}

\begin{thm} \textbf{(\cite{P}, Theorem 3.1.4 and Theorem 3.1.5)} \label{T2.10}
\begin{enumerate}
\item $\mathscr{H}$ is a finite-dimensional subalgebra of
$A_k^{\mathbb{Z}_2}(x)$ where $\mathscr{H}$ is as in Definition
\ref{D2.9}.
\item The signed partition algebra $\overrightarrow{A}_k^{\mathbb{Z}_2}(x)$ and
$\mathscr{H}$ are equal.
\end{enumerate}
\end{thm}

\begin{thm}\textbf{(\cite{P}, Theorem 3.1.7)}\label{T2.11}
\begin{enumerate}
  \item[(i)] The dimension of $A_k^{\mathbb{Z}_2}(x)$ is

  \centerline{$ \sum n_{\lambda} \underset{i=1}{\overset{t}{\prod}} \left( 2^{\lambda_i-1} + 1 \right)$}

  \NI where the sum is over the partition $\lambda = (\lambda_1, \cdots, \lambda_t) \vdash 2k$ and $n_{\lambda}$ is the number of diagrams $d \in \mathbb{P}_k(x)$ such that $\| d \| = \lambda = (\lambda_1, \cdots, \lambda_t)$ be the partition of $2k$, corresponding to the set partition $d,$ where $\lambda_i$ is the cardinality of the equivalence class.
  \item[(ii)]The dimension of the signed partition algebra $\overrightarrow{A}_k^{\mathbb{Z}_2}(x)$ is $$k! \ 2^k + \sum [(2^r-1)/ 2^r]^s \underset{i \geq 1}\prod(2^{\lambda_i-1} + 1)$$ where the sum is over the partition diagrams $d$ in $\mathbb{P}_k(x), \| d \| = \lambda = (\lambda_1, \lambda_2, \cdots, \lambda_t) \rightarrow 2k, \\ r = k - \sharp^p(d), s=0 $ if $|d^+| \neq k$ and $|d^-| \neq k, s=1$ if and only if $|d^+| = k$ or $|d^-| = k$ and $s=2$ if $d \not \in S_k, |d^+| = |d^- | = k.$

\end{enumerate}
\end{thm}

\begin{ex}\label{E2.12}
\begin{enumerate}
  \item[(i)] For $k = 1, 2,  \cdots $, Dimensions of $A_k^{\mathbb{Z}_2}(x)$ are $7, 164, \cdots $
  \item[(ii)] For $k = 1, 2, 3, \cdots $, Dimensions of $\overrightarrow{A}_k^{\mathbb{Z}_2}(x)$ are $3, 85, 5055, \cdots $

\end{enumerate}
\end{ex}

\begin{lem} \label{L2.13}

Let $I_{2k}^{2k}$ be as in Definition \ref{D2.9} then $I^{2k}_{2k}
\simeq \mathbb{Z}_2 \wr \mathfrak{S}_k.$
\end{lem}

\begin{proof}
Let $d \in I^{2k}_{2k}, $ then $\sharp(d) = 2k$ and $\sharp(R^d) =
k$ and $R^d$ is a permutation in $\mathfrak{S}_k.$

Define,

\centerline{$f(i) = \left\{
                      \begin{array}{ll}
                        \overline{1} , & \hbox{if $(i, e) \sim (i', g) $;} \\
                       \overline{0} , & \hbox{if $(i, e) \sim (i', e)$.}
                      \end{array}
                    \right.
$}

Thus, $d = (f, R^d) \in \mathbb{Z}_2 \wr \mathfrak{S}_k.$
\end{proof}

\begin{thm} \textbf{(\cite{RGA}, Theorem 3.26)} \label{T2.14}

Let $R$ be a commutative ring with unity. Let $\Lambda_{s_1} = \{(\lambda_1, \lambda_2) \ | \ \lambda_1 \vdash k_1, \lambda_2 \vdash k_2, k_1 + k_2 = s_1\}$ and $\Lambda_{s_2} = \{ \mu \ | \ \mu \vdash s_2\}.$ For $(\lambda_1, \lambda_2) \in \Lambda_{s_1}, $ and $\mu \in \Lambda_{s_2}$ define $M^{(\lambda_1, \lambda_2)}$ and $M^{\mu}$ be the set of all standard tableaux of shape $(\lambda_1, \lambda_2)$ and $\mu$
respectively.

\begin{enumerate}
    \item[(i)] The algebra $\mathscr{H} = R[\mathbb{Z}_2 \wr
\mathfrak{S}_{s_1}]$ is a free $R$-module with basis

\centerline{$\mathscr{M} = \left\{ m^{\lambda}_{s_{\lambda},
t_{\lambda} } \ | \ s_{\lambda} \text{ and } t_{\lambda} \text{
are standard tableaux of shape } \lambda \text{ for some
bi-partition } \lambda \text{ of } k \right.$}

 $\left.  \text{ in } M^{(\lambda_1,
\lambda_2)} \text{ and } \lambda = (\lambda_1, \lambda_2) \right\}$

\NI where $m^{\lambda}_{s_{\lambda}, t_{\lambda}}$ is as in Definition 3.14 of \cite{RGA}.

\NI Moreover, $\mathscr{M}$ is a cellular basis for $\mathscr{H}.$
    \item[(ii)] The algebra $\mathscr{H}' = R[
\mathfrak{S}_{s_2}]$ is a free $R$-module with basis

\centerline{$\mathscr{M}' = \big\{ m^{\mu}_{s_{\mu}, t_{\mu} } \ |
\ s_{\mu} \text{ and } t_{\mu} \text{ are standard tableaux of
shape } \mu \text{ for some  partition } \mu \text{ of } k \text{
in } M^{\mu} \big\}$}
\NI where $m^{\mu}_{s_{\mu}, t_{\mu}}$ is as in Definition 3.14 of \cite{RGA}.

\NI Moreover, $\mathscr{M}$ is a cellular basis for $\mathscr{H}'.$

\NI Also, $\mathscr{M}$ is a cellular basis for $\mathscr{H}' \otimes K(x),$ where $K$ is a field.
\end{enumerate}
\end{thm}

\begin{thm}\label{T2.15}
Let $R \left[\left(\mathbb{Z}_2 \wr \mathfrak{S}_{s_1} \right) \times \mathfrak{S}_{s_2} \right]$ be the $R$-algebra, then by Theorem \ref{T2.14}, $R[\left(\mathbb{Z}_2 \wr \mathfrak{S}_{s_1}\right) \times \mathfrak{S}_{s_2}] \simeq R[\mathbb{Z}_2 \wr \mathfrak{S}_{s_1}] \otimes R[\mathfrak{S}_{s_2}]$ is a cellular algebra with a cell datum $(\Lambda_{s_1, s_2}, M^{((\lambda_1, \lambda_2), \mu)}, C^{((\lambda_1, \lambda_2), \mu)}, \ast)$ given as follows:
\begin{enumerate}
 \item [(i)] $\Lambda_{s_1, s_2}:=\{((\lambda_1, \lambda_2), \mu) \ | \ |\lambda_1| + |\lambda_2| = s_1, \mu \vdash s_2 \} \cup \{ ((\lambda_1, \lambda_2), \Phi) \ | \ |\lambda_1| + |\lambda_2| = s_1\} \cup \{ ((\Phi, \Phi), \mu) \ | \\ \mu \vdash s_2\} \cup \{\Phi\}$ (ordered lexicographically) is a partially ordered set.
     \item [(ii)] $M^{((\lambda_1, \lambda_2), \mu)} := \{ ((s_{\lambda_1}, s_{\lambda_2}), s_{\mu}) \ | \ s_{\lambda_1}, s_{\lambda_2} \text{ and } s_{\mu} \text{ are the standard tableaus of shape } \lambda_1, \lambda_2 \\ \text{ and } \mu  \text{ respectively}\}$ such that

         \centerline{$C^{((\lambda_1, \lambda_2), \mu)} : \underset{\lambda, \mu \in \Lambda}{\coprod} M^{((\lambda_1, \lambda_2), \mu)} \times M^{((\lambda_1, \lambda_2), \mu)} \rightarrow \left(\mathbb{Z}_2 \wr \mathfrak{S}_{s_1}\right) \times \mathfrak{S}_{s_2}$}
         \NI is an injective map with image an $R$ basis of $ \left(\mathbb{Z}_2 \wr \mathfrak{S}_{s_1} \right) \times \mathfrak{S}_{s_2}.$
  \item [(iii)] If $\lambda = (\lambda_1, \lambda_2)$ and $S = ((s_{\lambda_1}, s_{\lambda_2}), s_{\mu}), T = ((t_{\lambda_1}, t_{\lambda_2}), t_{\mu}) \in M^{((\lambda_1, \lambda_2), \mu)}, $ write

     \centerline{$C(S, T) = m^{\lambda}_{s_{\lambda} t_{\lambda}} m^{\mu}_{s_{\mu} t_{\mu}}$}
       \NI where $m^{\lambda}_{s_{\lambda} t_{\lambda}}$ and $ m^{\mu}_{s_{\mu} t_{\mu}}$ are as in Theorem \ref{T2.14}. $\ast $ is the anti-automorphism of

        \NI $(\mathbb{Z}_2 \wr \mathfrak{S}_{s_1}) \times \mathfrak{S}_{s_2}$ such that $((f, \sigma_1), \sigma_2)^{\ast} = ((f, \sigma_1)^{\ast}, \sigma_2^{\ast})= ((f, \sigma_1)^{-1}, \sigma_2^{-1}) \ \ \forall \ ((f, \sigma_1), \sigma_2) \in \left( \mathbb{Z}_2 \wr \mathfrak{S}_{s_1}\right) \times \mathfrak{S}_{s_2}$ such that $\left(C^{((\lambda_1, \lambda_2), \mu)}(S, T)\right)^{\ast} = C^{((\lambda_1, \lambda_2), \mu)}(T, S).$
  \end{enumerate}
\end{thm}

\section{\textbf{Differences between the algebras}}

In this section, we illustrate that the algebra of $\mathbb{Z}_2$-relations $A_k^{\mathbb{Z}_2}(x)$ and signed partition algebras $\overrightarrow{A}_k^{\mathbb{Z}_2}(x)$ are different from the $\mathbb{Z}_2$-colored partition algebra $P_k(x;\mathbb{Z}_2)$ introduced  in \cite{M} and Tanabe algebras $T_{k, m}(x)$ introduced in \cite{T}.

\begin{ex} \label{E3.1}
This example clearly illustrates that the signed partition algebras are different from $\mathbb{Z}_2$-colored partition algebra introduced in \cite{M}.

\begin{center}
\includegraphics[height=11cm, width=15cm]{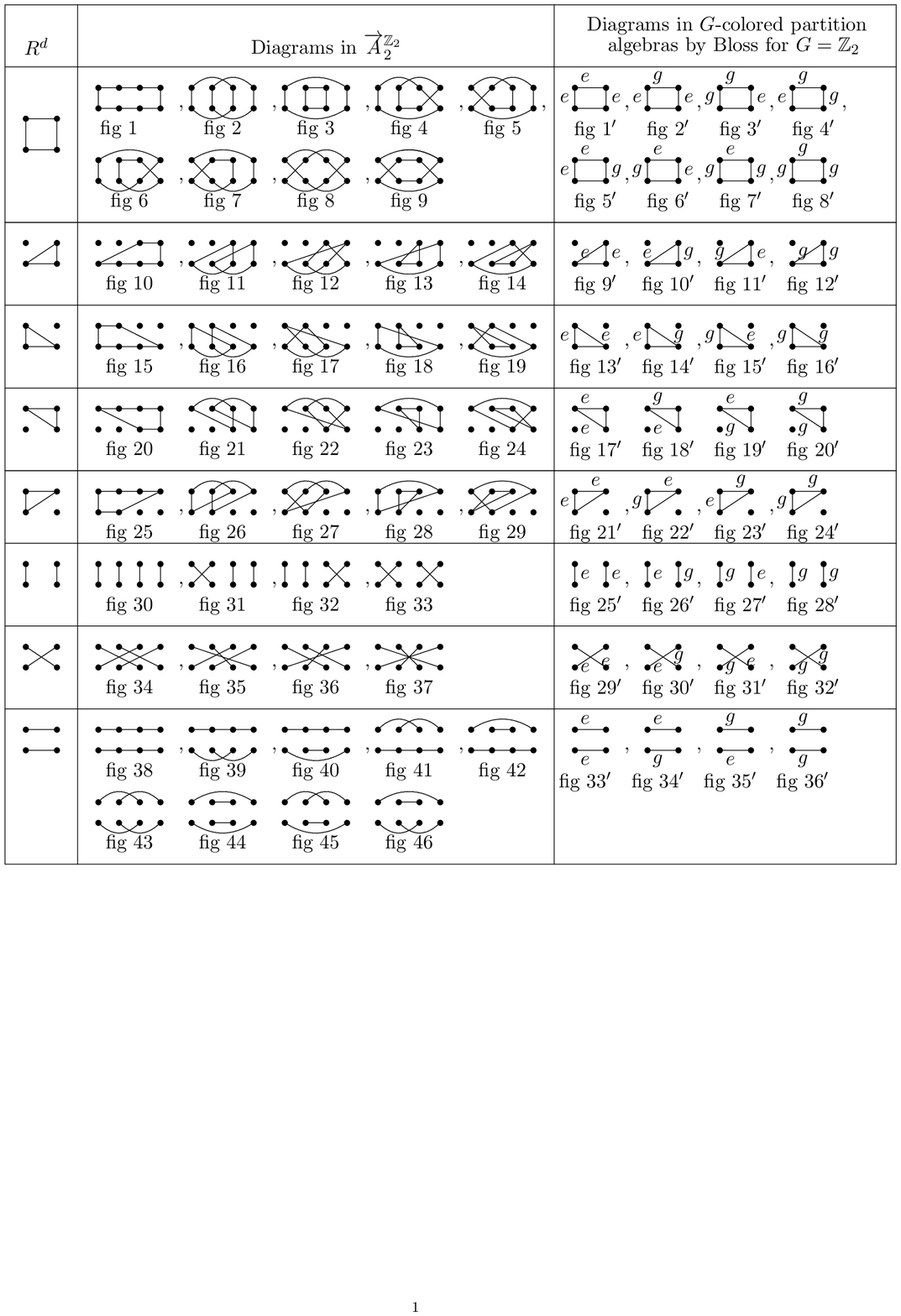}
\end{center}

\begin{center}
\vspace{-1cm}
\includegraphics[height=11cm, width=15.3cm]{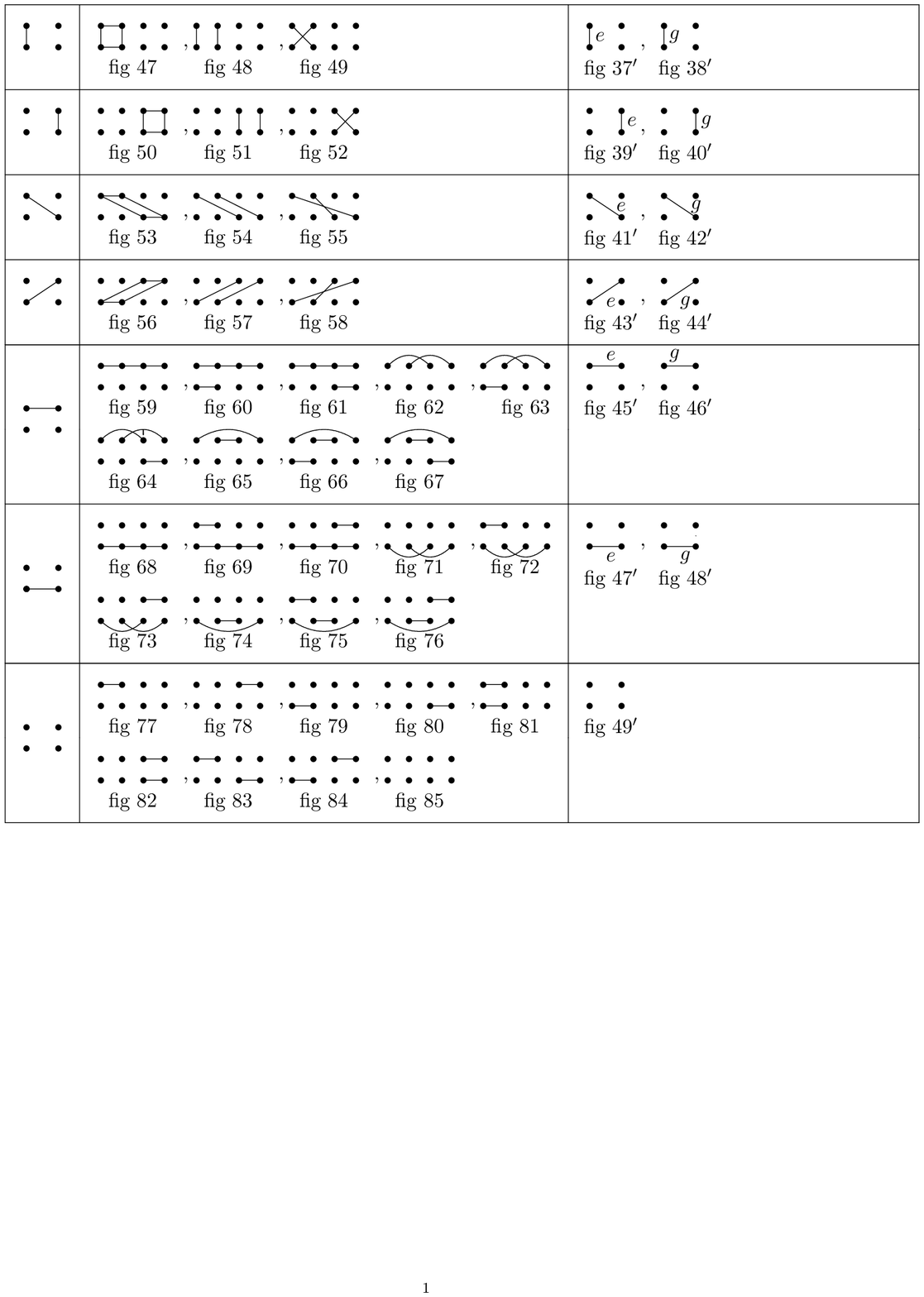}
\end{center}
\end{ex}

\begin{note}\label{N1}
In the algebra of $\mathbb{Z}_2$-relations and signed partition algebras, the set of all diagrams having no horizontal edges and each through class contains two vertices is isomorphic to the hyperoctahedral group of type $B_n$ whereas in Tanabe algebras, the set of all diagrams having no horizontal edges and each through class contains two vertices is isomorphic to the symmetric group.

Thus, the representations of algebra of $\mathbb{Z}_2$ relations and signed partition algebras are determined by the representations of hyperoctahedral group of type $B_n$ whereas  the representations of Tanabe algebras are determined by the representations of symmetric group.
\end{note}

\section{\textbf{The algebra of $\mathbb{Z}_2$ relations and signed partition algebras as Tabular algebras}}

In this section, we realize the algebra of $\mathbb{Z}_2$ relations and signed partition algebras as tabular algebras introduced  in \cite{GM1}.

\begin{notn} \label{N4.1}
Let $d \in I^{2k}_{2s_1+s_2}\left(\widetilde{d} \in \widetilde{I}^{2k}_{2s_1+s_2}\right),$ be as in Definition \ref{D2.5}$\left( \ref{D2.7}\right)$.
\begin{enumerate}
  \item [(i)]The vertex having least integer value in a connected component of $d(\widetilde{d})$ is called the minimal vertex of the connected component.
\item[(ii)] $|d|(|\widetilde{d}|)$ denotes the number of connected components in $d(\widetilde{d}).$
  \end{enumerate}
\end{notn}

\begin{defn} \label{D4.2}

Define,
\begin{enumerate}
  \item[(i)] $M[(r,(s_1, s_2))] = \Big\{(d, P) \ | \ d \in R_k^{\mathbb{Z}_2}, P \in R_{k'}^{\mathbb{Z}_2} \text{ and } d \setminus P \in R_{k-k'}^{\mathbb{Z}_2}, |d| \geq
2s_1 + s_2, P \text{ is a subset of the
}  $

\NI $ \text{  set of all connected components of } d \text{ with } |P| = 2s_1 + s_2 \text{ where } r = 2s_1 + s_2,$

\NI $ P = P_1^e \cup P_1^g \cup \cdots \cup P_{s_1}^e \cup
P_{s_1}^g \cup P_1^{\mathbb{Z}_2} \cup \cdots \cup
P_{s_2}^{\mathbb{Z}_2}$ such that $H^d_{R^{P_i^{\{e\}}}} = \{e\},
1 \leq i \leq s_1,$ \\ $H^d_{R^{P_j^{\mathbb{Z}_2}}} = \mathbb{Z}_2, 1 \leq
j \leq s_2 \Big\}.$

  \item[(ii)] $\widetilde{M}[(r,(s_1, s_2))] = \Big\{(\widetilde{d}, \widetilde{P}) \ | \ \widetilde{d} \in R_k^{\mathbb{Z}_2}, \widetilde{P} \in R_{k'}^{\mathbb{Z}_2} \text{ and } \widetilde{d} \setminus \widetilde{P} \in R_{k-k'}^{\mathbb{Z}_2}, |d| \geq
2s_1 + s_2, \widetilde{P} \text{ is a subset of the
}  $

\NI $ \text{  set of all connected components of } \widetilde{d} \text{ with } |\widetilde{P}| = 2s_1 + s_2 \text{ where } r = 2s_1 + s_2, $

\NI $ \widetilde{P} = \widetilde{P}_1^e \cup \widetilde{P}_1^g \cup \cdots \cup \widetilde{P}_{s_1}^e \cup
\widetilde{P}_{s_1}^g \cup \widetilde{P}_1^{\mathbb{Z}_2} \cup \cdots \cup
\widetilde{P}_{s_2}^{\mathbb{Z}_2}$ such that $H^d_{R^{\widetilde{P}_i^{\{e\}}}} = \{e\},
1 \leq i \leq s_1,$ \\ $H^d_{R^{\widetilde{P}_j^{\mathbb{Z}_2}}} = \mathbb{Z}_2, 1 \leq
j \leq s_2$ and $2r_1(r_2)$ is the number of $\{e\}(\mathbb{Z}_2)$ connected components in $\widetilde{d} \setminus \widetilde{P}, \\ s_1+s_2+r_1+r_2 \leq k -1$ if $s_1+s_2+r_1+r_2 = k$ then $s_1 = k$ or $r_1 \neq 0\Big\}.$
\end{enumerate}

\end{defn}

\NI  We shall now introduce an ordering for the connected components in
$P.$

\NI Suppose that  $P = P_1^e \cup P_1^g \cup \cdots \cup P_{s_1}^e \cup P_{s_1}^g
\cup P_1^{\mathbb{Z}_2} \cup \cdots \cup P_{s_2}^{\mathbb{Z}_2}$
then $R^P = R^{P_1^{\{e\}}} \cup \cdots \cup R^{P_{s_1}^{\{e\}}}
\cup R^{P_1^{\mathbb{Z}_2}} \cup \cdots \cup
R^{P_{s_2}^{\mathbb{Z}_2}}.$

Let $a_{11}, \cdots, a_{1s_1}$ be the minimal vertices of the
connected components $R^{P_1^{\{e\}}}, \cdots,
R^{P_{s_1}^{\{e\}}}$ in $R^P$  and $b_{11}, \cdots, b_{1s_2}$ be
the minimal vertices of the connected components
$R^{P_1^{\mathbb{Z}_2}}, \cdots, R^{P_{s_2}^{\mathbb{Z}_2}}$ in
$R^P$ then

\centerline{ $P_i^e < P_j^e$ \text{ and } $P_i^g < P_j^g$ \text{ if
and only if } $R^{P_i^{\{e\}}} < R^{P_j^{\{e\}}}$ \text{ if and
only if } $a_{1i} < a_{1j} \in R^{P}$ \text{ and }}

\centerline{ $P_l^{\mathbb{Z}_2} < P_f^{\mathbb{Z}_2}$ \text{ if
and only if } $R^{P_l^{\mathbb{Z}_2}} < R^{P_f^{\mathbb{Z}_2}}$
\text{ if and only if } $b_{1l} < b_{1f} \in R^P.$}

\NI Similarly, we can introduce an ordering for the connected components in
$\widetilde{P}$ as in $P.$

\begin{lem}\label{L4.4}
Let $M[(r, (s_1, s_2))]$ and $\widetilde{M}[(r, (s_1, s_2))]$ be as in Definition \ref{D4.2}.
\begin{enumerate}
  \item[(i)] Each $d \in I^{2k}_{2s_1+s_2}$ can be associated with a pair of
elements $(d^+, P),  (d^{-}, Q) \in M[(r, (s_1, s_2))]$ and \\ an element
$((f, \sigma_1), \sigma_2) \in \left(\mathbb{Z}_2 \wr \mathfrak{S}_{s_1}\right)
\times \mathfrak{S}_{s_2}$ where $(d^+, P), (d^-, Q) \in M[(r, (s_1, s_2))]$ and \\ $((f, \sigma_1), \sigma_2) \in (\mathbb{Z}_2 \wr \mathfrak{S}_{s_1}) \times \mathfrak{S}_{s_2}.$
  \item[(ii)]  Each $\widetilde{d} \in \widetilde{I}^{2k}_{2s_1+s_2}$ can be associated with a pair of
elements $(\widetilde{d}^+, \widetilde{P}),  (\widetilde{d}^{-}, \widetilde{Q}) \in \widetilde{M}[(r, (s_1, s_2))]$ and \\ an element
$((\widetilde{f}, \widetilde{\sigma}_1), \widetilde{\sigma}_2) \in \left(\mathbb{Z}_2 \wr \mathfrak{S}_{s_1}\right)
\times \mathfrak{S}_{s_2}$ where $(\widetilde{d}^+, \widetilde{P}), (\widetilde{d}^-, \widetilde{Q})\in \widetilde{M}[(r, (s_1, s_2))]$ and \\ $((\widetilde{f}, \widetilde{\sigma}_1), \widetilde{\sigma}_2)\in (\mathbb{Z}_2 \wr \mathfrak{S}_{s_1}) \times \mathfrak{S}_{s_2}.$
\end{enumerate}

\end{lem}

\begin{proof}
\NI \textbf{Proof of (i):}Let $d \in  I^{2k}_{2s_1+s_2}.$

$d^+, d^-$ are the diagrams obtained from the diagram $d$ by
restricting the vertex set to

$\{(1, e), (1, g), \cdots, (k, e), (k, g)\}$ and $\{(1', e), (1',
g), \cdots, (k', e), (k', g)\}$ respectively.

Identifying $\{(1',e), (1', g), \cdots, (k', e), (k', g)\}$ with
$\{(1, e), (1, g), \cdots, (k, e), (k, g)\}$ by sending

\centerline{$(i', e) \mapsto (i, e)$ \text{ and } $(i', g) \mapsto
(i, g).$}

Thus, $d^+, d^- \in R_k^{\mathbb{Z}_2}.$

Let $S_d $ be the set of all through classes of $d.$ Let $P$
denote the set of all  connected components obtained from $S_d$ by restricting
the vertex set to $\{(1, e), (1, g), \cdots, (k, e), (k, g)\}.$ i.e., $S_d \cap d^+ = P.$

Thus, $|P| = 2s_1 + s_2.$

Similarly, let $Q$ denote the set of all connected components obtained from
$S_d$ by restricting the vertex set to $\{(1', e), (1', g),
\cdots, (k', e), (k', g)\}.$ i.e., $S_d \cap d^- = Q.$

Identify $\{(1', e), (1', g),
\cdots, (k', e), (k', g)\}$ with $\{(1, e), (1, g), \cdots, (k,
e), (k, g)\}$ by sending

\centerline{$(i', e) \mapsto (i, e) \text{ and } (i', g) \mapsto
(i, g)$.}

Thus, $|Q| = 2s_1 + s_2.$

Write

\centerline{ $P = P_1^e \ \cup \ P_1^g \ \cup \ \cdots \ P_{s_1}^e \ \cup \ P_{s_1}^g \
\cup \ P_1^{\mathbb{Z}_2} \ \cup \ \cdots \ P_{s_2}^{\mathbb{Z}_2}$ and }

\centerline{$Q = Q_1^e \ \cup \ Q_1^g \ \cup \ \cdots \ \cup \ Q_{s_1}^e \ \cup \
Q_{s_1}^g \ \cup \ Q_1^{\mathbb{Z}_2} \ \cup \ \cdots \
Q_{s_2}^{\mathbb{Z}_2}$}

Define an element $(f, \sigma_1)$ as follows:

If there is a connected component $X \in S_d$ containing $P_i^e$
and $Q_j^{g'}, g' \in \mathbb{Z}_2$

\NI then, define $\sigma_1(i) = j$ and

\centerline{$f(i) = \left\{%
\begin{array}{ll}
    \overline{1}, & \hbox{ if $g' = g$;} \\
    \overline{0}, & \hbox{if $g' = e.$} \\
\end{array}%
\right.$}

Thus, $(f, \sigma_1) \in \mathbb{Z}_2 \wr \mathfrak{S}_{s_1}.$

Similarly, define $\sigma_2$ as follows:

If there is a connected component $Y \in S_d$ containing
$P_l^{\mathbb{Z}_2}$ and $Q_{m}^{\mathbb{Z}_2}$ then, define $\sigma_2(l) = m.$

Thus, $\sigma_2 \in \mathfrak{S}_{s_2}$ which implies that  $\left((f, \sigma_1), \sigma_2\right) \in \left(\mathbb{Z}_2 \wr \mathfrak{S}_{s_1}\right) \times \mathfrak{S}_{s_2}.$

\NI \textbf{Proof of (ii):} By Definition \ref{D2.8}, $s_1+s_2 + H_e(\widetilde{d}^+) +H_{\mathbb{Z}_2}(\widetilde{d}^+) \leq k-1$ and $s_1+s_2 + H_e(\widetilde{d}^-) +H_{\mathbb{Z}_2}(\widetilde{d}^-) \leq k-1$ and the proof of (ii) is same as proof of (i).
\end{proof}

\begin{lem} \label{L4.5}
\begin{enumerate}
  \item[(i)] For every pair $(d', P), (d'', Q) \in M[(r,(s_1, s_2))]$ and an element \\ $((f, \sigma_1), \sigma_2) \in \left(\mathbb{Z}_2 \wr \mathfrak{S}_{s_1}\right)
\times \mathfrak{S}_{s_2}$ there is a unique diagram $d \in I^{2k}_{2s_1+s_2}
$ where $d^+ = (d', P),  d^- = (d'', Q)$ such that there is a unique connected component of $d$ containing $P_i^e$ and $Q^{g'}_{\sigma_1(i)}$ and $P_{j}^{\mathbb{Z}_2}$ and $Q_{\sigma_2(j)}^{\mathbb{Z}_2}.$
  \item[(ii)] For every pair $(\widetilde{d}', \widetilde{P}), (\widetilde{d}'', \widetilde{Q}) \in \widetilde{M}[(r,(s_1, s_2))]$ and an element  $((\widetilde{f}, \widetilde{\sigma}_1), \widetilde{\sigma}_2) \in \left(\mathbb{Z}_2 \wr \mathfrak{S}_{s_1}\right)
\times \mathfrak{S}_{s_2}$ there is a unique diagram $\widetilde{d} \in \widetilde{I}^{2k}_{2s_1+s_2}
$ where $\widetilde{d}^+ = (\widetilde{d}', \widetilde{P}),  \widetilde{d}^- = (\widetilde{d}'', \widetilde{Q})$ such that there is a unique connected component of $\widetilde{d}$ containing $\widetilde{P}_i^e$ and $\widetilde{Q}^{g'}_{\sigma_1(i)}$ and $\widetilde{P}_{j}^{\mathbb{Z}_2}$ and $\widetilde{Q}_{\sigma_2(j)}^{\mathbb{Z}_2}.$

\end{enumerate}
\end{lem}

\begin{proof}
\NI \textbf{Proof of (i):}
Let $P = P_1^e \cup P_1^g \cup \cdots \cup P_{s_1}^e \cup
P_{s_1}^g \cup P_1^{\mathbb{Z}_2} \cup \cdots
P_{s_2}^{\mathbb{Z}_2}$ and $Q = Q_1^e \cup Q_1^g \cup \cdots \cup
Q_{s_1}^e \cup Q_{s_1}^g \cup Q_1^{\mathbb{Z}_2} \cup \cdots \cup
Q_{s_2}^{\mathbb{Z}_2}.$

Let $\{a_{11}^e, \cdots, a_{1s_1}^e\}, \{a_{11}^g, \cdots,
a_{1s_1}^g\}$ and $\{b^e_{11}, \cdots, b^e_{1s_2}\}$ be the minimal
vertices of the connected components $\{R^{P_1^e}, \cdots,
R^{P_{s_1}^e}\}, \{R^{P_1^g}, \cdots, R^{P_{s_1}^g} \}$ and
$\{R^{P_1^{\mathbb{Z}_2}}, \cdots, R^{P_{s_2}^{\mathbb{Z}_2}}\}$
respectively.

Similarly, let $\{l_{11}^e, \cdots, l_{1s_1}^e\}, \{l_{11}^g,
\cdots, l_{1s_1}^g\}$ and $\{f^e_{11}, \cdots, f^e_{1s_2}\}$ be the
minimal vertices of the connected components $\{R^{Q_1^e}, \cdots,
R^{Q_{s_1}^e}\}, \{R^{Q_1^g}, \cdots, R^{Q_{s_1}^g} \}$ and
$\{R^{Q_1^{\mathbb{Z}_2}}, \cdots, R^{Q_{s_2}^{\mathbb{Z}_2}}\}$
respectively.

Let $d \in \widetilde{I}^{2k}_{2s_1+s_2}$ be obtained as follows:

\begin{enumerate}
    \item[(i)] Draw $(d', P)$ above $(d'', Q).$
    \item[(ii)]Connect $P_i^e$ to $Q^{g'}_{\sigma_1(i)}$ if $f(i) = g'.$ Also, connect $P_j^{\mathbb{Z}_2}$ to $Q^{\mathbb{Z}_2}_{\sigma_2(j)}.$
    \item[(iii)] All other connected components in $(d', P)\left((d'', Q)\right)$ other than the connected
components of $P(Q)$ will remain as horizontal edges or isolated
points in the top(bottom) row of $d \in \widetilde{I}^{2k}_{2s_1+s_2}, $ by our construction $d^+ = (d', P)$ and $d^- = (d'', Q).$
\end{enumerate}

\NI \textbf{Proof of (ii):} Proof of (ii) is similar to the proof of (i).
\end{proof}

\begin{rem}\label{R4.6}
By Lemma \ref{L4.4}, any $d \in I_{2k}(\widetilde{d} \in \widetilde{I}_{2k}), $ is denoted by $C^{((f, \sigma_1), \sigma_2)}_{(d^+, P), (d^-, Q)}\left( \widetilde{C}^{((\widetilde{f}, \widetilde{\sigma}_1), \widetilde{\sigma}_2)}_{(\widetilde{d}^+, \widetilde{P}), (\widetilde{d}^-, \widetilde{Q})}\right).$
\end{rem}

\begin{defn}\label{D4.7}
\begin{enumerate}
  \item[(i)] Define a map $\phi^r_{s_1, s_2}: M[(r,(s_1, s_2))] \times M[(r,(s_1, s_2))]
\rightarrow R [\left(\mathbb{Z}_2 \wr \mathfrak{S}_{s_1}\right)
\times \mathfrak{S}_{s_2}]$ as follows:

\centerline{$\phi^r_{s_1, s_2}\left((d', P), (d'', Q)\right) = x^{l(P \vee Q)} ((f, \sigma_1),
    \sigma_2)$ and }

  \item[(ii)] Define a map $\widetilde{\phi}^r_{s_1, s_2}: \widetilde{M}[(r,(s_1, s_2))] \times \widetilde{M}[(r,(s_1, s_2))]
\rightarrow R [\left(\mathbb{Z}_2 \wr \mathfrak{S}_{s_1}\right)
\times \mathfrak{S}_{s_2}]$ as follows:

\centerline{$\widetilde{\phi}^r_{s_1, s_2}\left((\widetilde{d}', \widetilde{P}), (\widetilde{d}'', \widetilde{Q})\right) = x^{l(\widetilde{P} \vee \widetilde{Q})} ((\widetilde{f}, \widetilde{\sigma}_1),
    \widetilde{\sigma}_2)$  }

\end{enumerate}
\NI if
\begin{enumerate}
    \item[(a)] No two connected components of $Q(\widetilde{Q})$ in $d''(\widetilde{d}'')$ have non-empty intersection with a common \\ connected component of $d'(\widetilde{d}')$ in $d'. d''(\widetilde{d}'.\widetilde{d}'')$, or vice versa.
       \item[(b)] No connected component of $Q(\widetilde{Q})$ has non-empty intersection only with the connected components   excluding the connected components of $P(\widetilde{P})$ in $d' . d''(\widetilde{d}'.\widetilde{d}'').$ Similarly, no connected component in $P(\widetilde{P})$ has non-empty intersection only with  a connected component excluding the connected \\ components of $Q(\widetilde{Q})$ in $d'. d''(\widetilde{d}'.\widetilde{d}'').$
        \end{enumerate}
        where $l(P \vee Q)\left(l(\widetilde{P} \vee \widetilde{Q}) \right)$ denotes the number of connected
    components in $d' . d''(\widetilde{d}'.\widetilde{d}'')$ excluding the union of all the
    connected components of $P(\widetilde{P})$ and $Q(\widetilde{Q}).$

    The permutation $\left((f, \sigma_1), \sigma_2 \right)\left(\left((\widetilde{f}, \widetilde{\sigma}_1), \widetilde{\sigma}_2 \right) \right)$ is obtained as
    follows: If there is a unique connected \\ component in $d'. d''(\widetilde{d}'.\widetilde{d}'')$
    containing $P_i^e(\widetilde{P}_i^e)$ and $Q_j^{g'}(\widetilde{Q}_j^{g'})$ then, define $\sigma_1(i) = j(\widetilde{\sigma}_1(i) = j)$ and

    \centerline{$f(i) = \widetilde{f}(i) = \left\{%
\begin{array}{ll}
    \overline{1}, & \hbox{if $g' = g$;} \\
    \overline{0}, & \hbox{if $g' = e$.} \\
\end{array}%
\right.   $}

Also, if there is a unique connected component in $d' . d''(\widetilde{d}'.\widetilde{d}'')$
containing $P_l^{\mathbb{Z}_2}$ and $Q_f^{\mathbb{Z}_2}(\widetilde{P}_l^{\mathbb{Z}_2}$ and $\widetilde{Q}_f^{\mathbb{Z}_2})$ then, define $\sigma_2(l) = f(\widetilde{\sigma}_2(l) = f).$

\NI Otherwise, $\phi^r_{s_1, s_2}\left((d', P), (d'', Q)\right) = 0 \left(\widetilde{\phi}^r_{s_1, s_2}\left((\widetilde{d}', \widetilde{P}), (\widetilde{d}'', \widetilde{Q})\right) = 0 \right).$
\end{defn}

 Since the algebra of $\mathbb{Z}_2$-relations and signed partition algebras are subalgebras of partition algebras the proof of Lemmas \ref{L4.8} and \ref{L4.9} follow as in \cite{X}.

\begin{lem}\label{L4.8}
\begin{enumerate}
  \item[(i)] Let $\mu, \nu \in I^{2k}_{2s_1 + s_2}$ then
$\sharp^p (\mu \nu) \leq 2s_1 +s_2 . \text{  If } \sharp^p(\mu \nu) = 2s_1 + s_2$ then

$$\mu \nu = C^{r_{\mu}[(d, R), (d'', Q)]((f', \sigma'_1), \sigma'_2)}_{((d, R), (d''', T))}$$ where $\mu = C^{((f, \sigma_1), \sigma_2)}_{(d, R),(d', P)}, \nu = C^{((f', \sigma'_1), \sigma'_2)}_{(d'', Q), (d''', T)}, (d, R), (d', P), (d'', Q), (d''', T) \in M[(r,(s_1, s_2))]$, $((f, \sigma_1), \sigma_2), $ $((f', \sigma'_1), \sigma'_2) \in \left( \mathbb{Z}_2 \wr \mathfrak{S}_{s_1} \right) \times \mathfrak{S}_{s_2}$ ,  $r_{\mu}[(d, R), (d'', Q)] = ((f, \sigma_1), \sigma_2) \phi^r_{s_1, s_2}[(d', P), (d'', Q)]$  and \\ $r_{\mu}[(d, R), (d'', Q)]$ is independent of $(d''', T)$ and $((f', \sigma'_1), \sigma'_2).$

  \item[(ii)] Let $\widetilde{\mu}, \widetilde{\nu} \in \widetilde{I}^{2k}_{2s_1 + s_2}$ then
$\sharp^p (\widetilde{\mu} \widetilde{\nu}) \leq 2s_1 +s_2 . \text{  If } \sharp^p(\widetilde{\mu} \widetilde{\nu}) = 2s_1 + s_2$ then

$$\widetilde{\mu} \widetilde{\nu} = C^{r_{\widetilde{\mu}}[(\widetilde{d}, \widetilde{R}), (\widetilde{d}'', \widetilde{Q})]((\widetilde{f}', \widetilde{\sigma}'_1), \widetilde{\sigma}'_2)}_{((\widetilde{d}, \widetilde{R}), (\widetilde{d}''', \widetilde{T}))}$$ where $\widetilde{\mu} = C^{((\widetilde{f}, \widetilde{\sigma}_1), \widetilde{\sigma}_2)}_{(\widetilde{d}, \widetilde{R}),(\widetilde{d}', \widetilde{P})}, \widetilde{\nu} = \widetilde{C}^{((\widetilde{f}', \widetilde{\sigma}'_1), \widetilde{\sigma}'_2)}_{(\widetilde{d}'', \widetilde{Q}), (\widetilde{d}''', \widetilde{T})}, (\widetilde{d}, \widetilde{R}), (\widetilde{d}', \widetilde{P}), (\widetilde{d}'', \widetilde{Q}), (\widetilde{d}''', \widetilde{T}) \in \widetilde{M}[(r,(s_1, s_2))]$, $((\widetilde{f}, \widetilde{\sigma}_1), \widetilde{\sigma}_2), $ $((\widetilde{f}', \widetilde{\sigma}'_1), \widetilde{\sigma}'_2) \in \left( \mathbb{Z}_2 \wr \mathfrak{S}_{s_1} \right) \times \mathfrak{S}_{s_2}$ ,  $r_{\widetilde{\mu}}[(\widetilde{d}, \widetilde{
R}), (\widetilde{d}'', \widetilde{Q})] = ((\widetilde{f}, \widetilde{\sigma}_1), \widetilde{\sigma}_2) \widetilde{\phi}^r_{s_1, s_2}[(\widetilde{d}', \widetilde{P}), (\widetilde{d}'', \widetilde{Q})]$  and \\ $r_{\widetilde{\mu}}[(\widetilde{d}, \widetilde{R}), (\widetilde{d}'', \widetilde{Q})]$ is independent of $(\widetilde{d}''', \widetilde{T})$ and $((\widetilde{f}', \widetilde{\sigma}'_1), \widetilde{\sigma}'_2).$

\end{enumerate}
\end{lem}

\begin{proof}

\NI If  $\sharp^p(\mu \nu) = 2s_1 + s_2,$ then the proof of (i) and (ii) follows from the definition of multiplication of partition algebras and Definition \ref{D4.7} and from Lemma 4.4 of \cite{X}.

\end{proof}

\begin{lem}\label{L4.9}
\begin{enumerate}
  \item[(i)] Let $\mu \in I^{2k}_{2s'_1+s'_2}, \nu \in I^{2k}_{2s_1 +s_2}$ then $\sharp^p (\mu \nu) \leq 2s_1 +s_2.$ If $\sharp^p(\mu \nu) = 2s_1 +s_2$ then $$\mu \nu = x^{l(P \vee Q)} C_{(w, F), (d''', T)}^{r_{\mu}[(w, F), (d'', Q)]((f', \sigma'_1), \sigma'_2)}$$ where  $\mu = C^{((f, \sigma_1), \sigma_2)}_{(d, R),(d', P)}, \nu = C^{((f', \sigma'_1), \sigma'_2)}_{(d'', Q), (d''', T)}, (d, R), (d', P) \in M[(r',(s'_1, s'_2))], (w, F), (d'', Q), \\ (d''', T) \in M[(r,(s_1, s_2))]$, $((f, \sigma_1), \sigma_2) \in \left(\mathbb{Z}_2 \wr \mathfrak{S}_{s'_1}\right) \times \mathfrak{S}_{s'_2}, $ $((f', \sigma'_1), \sigma'_2) \in \left( \mathbb{Z}_2 \wr \mathfrak{S}_{s_1} \right) \times \mathfrak{S}_{s_2}, r_{\mu}[(w, F), (d'', Q)]$ is independent of $((f', \sigma'_1), \sigma'_2)$ and $(d''', T)$.

  \item[(ii)]Let $\widetilde{\mu} \in \widetilde{I}^{2k}_{2s'_1+s'_2}, \widetilde{\nu} \in \widetilde{I}^{2k}_{2s_1 +s_2}$ then $\sharp^p (\widetilde{\mu} \widetilde{\nu}) \leq 2s_1 +s_2.$ If $\sharp^p(\widetilde{\mu} \widetilde{\nu}) = 2s_1 +s_2$ then $$\widetilde{\mu} \widetilde{\nu} = x^{l(\widetilde{P} \vee \widetilde{Q})} C_{(\widetilde{w}, \widetilde{F}), \\ (\widetilde{d}''', \widetilde{T})}^{r_{\widetilde{\mu}}[(\widetilde{w}, \widetilde{F}), (\widetilde{d}'', \widetilde{Q})]((\widetilde{f}', \widetilde{\sigma}'_1), \widetilde{\sigma}'_2)}$$ where  $\widetilde{\mu} = \widetilde{C}^{((\widetilde{f}, \widetilde{\sigma}_1), \widetilde{\sigma_2})}_{(\widetilde{d}, \widetilde{R}),(\widetilde{d}', \widetilde{P})}, \widetilde{\nu} = \widetilde{C}^{((\widetilde{f}', \widetilde{\sigma}'_1), \widetilde{\sigma}'_2)}_{(\widetilde{d}'', \widetilde{Q}), (\widetilde{d}''', \widetilde{T})}, (\widetilde{d}, \widetilde{R}), (\widetilde{d}', \widetilde{P}) \in \widetilde{M}[(r',(s'_1, s'_2))], (\widetilde{w}, \
widetilde{F}), (\widetilde{d}'', \widetilde{Q}), (\widetilde{d}''', \widetilde{T}) \in \widetilde{M}[(r,(s_1, s_2))]$, $((\widetilde{f}, \widetilde{\sigma}_1), \widetilde{\sigma}_2) \in \left(\mathbb{Z}_2 \wr \mathfrak{S}_{s'_1}\right) \times \mathfrak{S}_{s'_2}, $ $((\widetilde{f}', \widetilde{\sigma}'_1), \widetilde{\sigma}'_2) \in \left( \mathbb{Z}_2 \wr \mathfrak{S}_{s_1} \right) \times \mathfrak{S}_{s_2}, r_{\widetilde{\mu}}[(\widetilde{w}, \widetilde{F}), (\widetilde{d}'', \widetilde{Q})]$ is independent of $((\widetilde{f}', \widetilde{\sigma}'_1), \widetilde{\sigma}'_2)$ and $(\widetilde{d}''', \widetilde{T})$.

\end{enumerate}
\end{lem}

\begin{proof}
\NI \textbf{Proof of (i):}If $\sharp^p \left(C^{((f, \sigma_1), \sigma_2)}_{(d, R), (d', P)} \ C^{((f', \sigma'_1), \sigma'_2)}_{(d'', Q), (d''', T)} \right) = 2s_1 +s_2$, then by Lemma \ref{L4.5} and \cite{X} there exists $(w, F), (v, Q) \in M[(r,(s_1, s_2))], ((f'', \sigma''_1), \sigma''_2) \in \mathbb{Z}_2 \wr \mathfrak{S}_{s_1} \times \mathfrak{S}_{s_2}.$

\centerline{$C^{((f, \sigma_1), \sigma_2)}_{(d, R), (d', P)} \ C^{((f', \sigma'_1), \sigma'_2)}_{(d'', Q), (d''', T)} = C^{r_{\mu}[(w, F), (d'', Q)] ((f', \sigma'_1), \sigma'_2)}_{(w, F), (v, Q)}$}
\NI where $\mu = C^{((f, \sigma_1), \sigma_2)}_{(d, R), (d', P)}$ $r_{\mu}[(w, F), (d'', Q)] ((f', \sigma'_1), \sigma'_2)  = x^{l(P \vee Q)} ((f'', \sigma''_1), \sigma''_2)$ and it is independent of $((f', \sigma'_1), \sigma'_2)$ and $(d''', T).$

\NI \textbf{Proof of (ii):} Proof of (ii) is same as that of proof of (i).
\end{proof}

\begin{defn}\label{D4.10}
Put,

\begin{enumerate}
  \item[(i)] $\Lambda = \{(r,(s_1, s_2)) \ | \ r = 2s_1 + s_2, 0 \leq s_1, s_2 \leq k\}$ and
  \item[(ii)] $\widetilde{\Lambda} = \{(r,(s_1, s_2)) \ | \ r = 2s_1 + s_2, 0 \leq s_1 \leq k, 0 \leq s_2 \leq k-1\}.$
\end{enumerate}

Define a relation '$\leq$' on $\Lambda(\widetilde{\Lambda})$ as follows:

\centerline{$(r,(s_1, s_2)) \leq (r', (s'_1, s'_2))$} if and only if
\begin{enumerate}
    \item[(a)] $r < r'$ or
    \item [(b)] $r' = r$ and $s_1 + s_2 < s'_1 + s'_2$
   \end{enumerate}

Thus $(\Lambda, \leq)((\widetilde{\Lambda}, \leq))$ is a partially ordered set.
\end{defn}

\begin{note}\label{N2}
Let $B^r(s_1, s_2)  = B(s_1, s_2)= \left(\mathbb{Z}_2 \wr \mathfrak{S}_{s_1}\right) \times
\mathfrak{S}_{s_2}$ and $\Gamma_{s_1, s_2} =
\mathscr{A}[ \left(\mathbb{Z}_2 \wr \mathfrak{S}_{s_1}\right) \times
\mathfrak{S}_{s_2}],$ where $r = 2s_1 + s_2.$ The elements of $B(s_1, s_2)$ forms a
basis of $\Gamma(s_1, s_2).$ Thus $(\Gamma(s_1, s_2), B(s_1,
s_2))$ is a hyper group.
\end{note}

\begin{defn}\label{D4.11}
Let $M[(r,(s_1, s_2))]\left(\widetilde{M}[(r,(s_1, s_2))] \right)$ be as in Definition \ref{D4.2}.

Define  maps,

\begin{enumerate}
  \item[(i)] $C : M[(r,(s_1, s_2))] \times B(s_1, s_2) \times M[(r,(s_1, s_2))]
 \rightarrow A^{\mathbb{Z}_2}_k$
as follows:

\centerline{$C \left[ (d', P), ((f, \sigma_1), \sigma_2), (d'', Q)\right] = d, $} where $d = C^{((f, \sigma_1), \sigma_2)}_{(d', P), (d'', Q)},$ as in Remark \ref{R4.6} and $d \in I_{2s_1+s_2}^{2k}.$

 By Lemma \ref{L4.5},
it is clear that $C$ is injective.

  \item[(ii)] $\widetilde{C} : \widetilde{M}[(r,(s_1, s_2))] \times B(s_1, s_2) \times \widetilde{M}[(r,(s_1, s_2))] \rightarrow \overrightarrow{A}^{\mathbb{Z}_2}_k,$
      as follows:

\centerline{$\widetilde{C} \left[ (\widetilde{d}', \widetilde{P}), ((\widetilde{f}, \widetilde{\sigma}_1), \widetilde{\sigma}_2), (\widetilde{d}'', \widetilde{Q})\right] = \widetilde{d}, $} where $\widetilde{d} = \widetilde{C}^{((\widetilde{f}, \widetilde{\sigma}_1), \widetilde{\sigma}_2)}_{(\widetilde{d}', \widetilde{P}), (\widetilde{d}'', \widetilde{Q})},$ as in Remark \ref{R4.6} and $\widetilde{d} \in \widetilde{I}_{2s_1+s_2}^{2k}.$

 By Lemma \ref{L4.5},
it is clear that $\widetilde{C}$ is injective.

\end{enumerate}

\end{defn}

\begin{defn}\label{D4.12}

Define,
\begin{enumerate}
  \item[(i)] $\ast : A^{\mathbb{Z}_2}_k \rightarrow
A^{\mathbb{Z}_2}_k$ as follows:

\centerline{$\left( C^{((f, \sigma_1), \sigma_2)}_{(d',P), (d'', Q)}\right)^{\ast} = \left( C^{((f, \sigma_1), \sigma_2)}_{(d',P), (d'', Q)}\right)^f = C_{(d'', Q), (d', P)}^{((f, \sigma_1), \sigma_2)^{-1}}$}
\NI where $f$ is the flip of the diagram and inverse mapping is the anti-automorphism of the hyper group $\left(\Gamma(s_1, s_2), B(s_1, s_2)\right).$

Clearly, $\ast$ is an involutary anti-automorphism of $A_k^{\mathbb{Z}_2}.$

  \item[(ii)]$\widetilde{\ast} : \overrightarrow{A}^{\mathbb{Z}_2}_k \rightarrow
\overrightarrow{A}^{\mathbb{Z}_2}_k$ as follows:

\centerline{$\left( \widetilde{C}^{((\widetilde{f}, \widetilde{\sigma}_1), \widetilde{\sigma}_2)}_{(\widetilde{d}',\widetilde{P}), (\widetilde{d}'', \widetilde{Q})}\right)^{\widetilde{\ast}} = \left( \widetilde{C}^{((\widetilde{f}, \widetilde{\sigma}_1), \widetilde{\sigma}_2)}_{(\widetilde{d}',\widetilde{P}), (\widetilde{d}'', \widetilde{Q})}\right)^f = \widetilde{C}_{(\widetilde{d}'', \widetilde{Q}), (\widetilde{d}', \widetilde{P})}^{((\widetilde{f}, \widetilde{\sigma}_1), \widetilde{\sigma}_2)^{-1}}$}
\NI where $f$ is the flip of the diagram and inverse mapping is the anti-automorphism of the hyper group $\left(\Gamma(s_1, s_2), B(s_1, s_2)\right).$

Clearly, $\ast$ is an involutary anti-automorphism of $\overrightarrow{A}_k^{\mathbb{Z}_2}.$

\end{enumerate}
\end{defn}

\begin{notn} \label{N4.13}
 If $b \in \Gamma(s_1, s_2)$ such that $b = \underset{((f_i,
\sigma_{1_i}), \sigma_{2_i}) \in B(s_1, s_2)}{\sum} c_i ((f_i,
\sigma_{1_i}), \sigma_{2_i})$ for some scalars.
\begin{enumerate}
  \item[(i)] We write $C^b_{(d', P), (d'', Q)} \in \mathscr{A} \left[A^{\mathbb{Z}_2}_k\right]$ as
shorthand for $\underset{((f_i, \sigma_{1_i}), \sigma_{2_i}) \in
B(s_1, s_2)}{\sum} c_i C^{((f_i, \sigma_{1_i}), \sigma_{2_i})}_{(d', P), (d'', Q)}.$

Also, write $C_{s_1, s_2}$ for the image under $C$ of $M[(r,(s_1, s_2))]
\times B(s_1, s_2) \times M[(r,(s_1, s_2))].$

  \item[(ii)]We write $\widetilde{C}^b_{(\widetilde{d}', \widetilde{P}), (\widetilde{d}'', \widetilde{Q})} \in \mathscr{A} \left[\overrightarrow{A}^{\mathbb{Z}_2}_k\right]$ as
shorthand for $\underset{((\widetilde{f}_i, \widetilde{\sigma}_{1_i}), \widetilde{\sigma}_{2_i}) \in
B(s_1, s_2)}{\sum} c_i \widetilde{C}^{((\widetilde{f}_i, \widetilde{\sigma}_{1_i}), \widetilde{\sigma}_{2_i})}_{(\widetilde{d}', \widetilde{P}), (\widetilde{d}'', \widetilde{Q})}.$

Also, write $\widetilde{C}_{s_1, s_2}$ for the image under $\widetilde{C}$ of $\widetilde{M}[(r,(s_1, s_2))]
\times B(s_1, s_2) \times \widetilde{M}[(r,(s_1, s_2))].$

\end{enumerate}
\end{notn}
\begin{thm}  \label{T4.14}
Let $\mathscr{A} = \mathbb{C}(x).$
\begin{enumerate}
  \item[(i)] An algebra of $\mathbb{Z}_2$-relations  $\mathscr{A}[A^{\mathbb{Z}_2}_{k}]$ is a tabular algebra  together with a table datum \\ $(\Lambda, \Gamma, B, M[(r, (s_1, s_2))], C, \ast)$ where :
\begin{enumerate}
  \item[(a)] $\Lambda$ is a finite poset where $\Lambda$ is as in Definition \ref{D4.10}. For each $(r, (s_1, s_2)) \in \Lambda$, $(\Gamma(s_1, s_2), B(s_1, s_2))$ is a hypergroup over $\mathbb{C}$ and $M[(r, (s_1, s_2))]$ is a finite set. The map

\centerline{$C: \underset{(r, (s_1, s_2)) \in \Lambda}{\coprod} (M[(r, (s_1, s_2))] \times B(s_1, s_2) \times M[(r, (s_1, s_2))]) \rightarrow A_k^{\mathbb{Z}_2}$} is injective with image an $\mathscr{A}$-basis of $A_k^{\mathbb{Z}_2.}$
  \item[(b)] $\ast$ is an $\mathscr{A}$-linear involutary anti-automorphism of $A_k^{\mathbb{Z}_2} .$

  \item[(c)] If $(r, (s_1, s_2)) \in \Lambda, ((f, \sigma_1), \sigma_2) \in \Gamma(s_1, s_2)$ and $(d', P), (d'', Q) \in M[(r, (s_1, s_2))] $ then for all $a \in A_k^{\mathbb{Z}_2}$ we have

\centerline{$a C^{((f, \sigma_1), \sigma_2)}_{(d', P), (d'', Q)} \equiv \underset{(d'''_i, R_i) \in M[(r, (s_1, s_2 ))]}{\sum} C^{r_a[(d'''_i, R_i), (d', P)]((f, \sigma_1), \sigma_2)}_{(d'''_i, R_i), (d'', Q)} \ \ \ \text{ mod } A_k^{\mathbb{Z}_2}(< (r, (s_1, s_2)))$,}
\NI where $r_a[(d'''_i, R_i), (d', P)] ((f, \sigma_1), \sigma_2)$ is independent of $(d'', Q)$ and of $((f, \sigma_1), \sigma_2).$

\end{enumerate}
  \item[(ii)] An algebra of signed partition algebras  $ \mathscr{A}[\overrightarrow{A}^{\mathbb{Z}_2}_{k}] $ is a tabular algebra  together with a table datum \\ $(\widetilde{\Lambda}, \Gamma, B, \widetilde{M}[(r, (s_1, s_2))], \widetilde{C}, \widetilde{\ast}) $ where :
\begin{enumerate}
  \item[(a)] $\widetilde{\Lambda}$ is a finite poset where $\widetilde{\Lambda}$ is as in Definition \ref{D4.10}. For each $(r, (s_1, s_2)) \in \widetilde{\Lambda}$, $(\Gamma(s_1, s_2), B(s_1, s_2))$ is a hypergroup over $\mathbb{C}$ and $ \widetilde{M}[(r, (s_1, s_2))]$ is a finite set. The map

\centerline{$\widetilde{C}: \underset{(r, (s_1, s_2)) \in \widetilde{\Lambda}}{\coprod} (\widetilde{M}[(r, (s_1, s_2))] \times B(s_1, s_2) \times \widetilde{M}[(r, (s_1, s_2))]) \rightarrow \overrightarrow{A}_k^{\mathbb{Z}_2}$} is injective with image an $\mathscr{A}$-basis of $\overrightarrow{A}_k^{\mathbb{Z}_2.}$
  \item[(b)]$\widetilde{\ast}$ is an $\mathscr{A}$-linear involutary anti-automorphism of $\overrightarrow{A}_k^{\mathbb{Z}_2}.$

  \item[(c)]If $(r, (s_1, s_2)) \in \widetilde{\Lambda}, ((f, \sigma_1), \sigma_2) \in \Gamma(s_1, s_2)$ and $ \left( (\widetilde{d}', \widetilde{P}), (\widetilde{d}'', \widetilde{Q}) \in \widetilde{M}[(r, (s_1, s_2))] \right)$ then for all $\widetilde{a} \in \overrightarrow{A}_k^{\mathbb{Z}_2}$ we have

\centerline{$\widetilde{a} \widetilde{C}^{((\widetilde{f}, \widetilde{\sigma}_1), \widetilde{\sigma}_2)}_{(\widetilde{d}', \widetilde{P}), (\widetilde{d}'', \widetilde{Q})} \equiv \underset{(\widetilde{d}'''_i, \widetilde{R}_i) \in \widetilde{M}[(r, (s_1, s_2 ))]}{\sum} \widetilde{C}^{r_{\widetilde{a}}[(\widetilde{d}'''_i, \widetilde{R}_i), (\widetilde{d}', \widetilde{P})]((\widetilde{f}, \widetilde{\sigma}_1), \widetilde{\sigma}_2)}_{(\widetilde{d}'''_i, \widetilde{R}_i), (\widetilde{d}'', \widetilde{Q})} \ \ \ \text{ mod } \overrightarrow{A}_k^{\mathbb{Z}_2}(< (r, (s_1, s_2)))$,}
  \NI where $r_{\widetilde{a}}[(\widetilde{d}'''_i, \widetilde{R}_i), (\widetilde{d}', \widetilde{P})] ((\widetilde{f}, \widetilde{\sigma}_1), \widetilde{\sigma}_2)$ is independent of $(\widetilde{d}'', \widetilde{Q})$ and of $((\widetilde{f}, \widetilde{\sigma}_1), \widetilde{\sigma}_2).$

\end{enumerate}

   \end{enumerate}
\end{thm}

\begin{proof}
The proof of (i)(a) and (ii)(a) follows Definitions \ref{D4.2}, \ref{D4.10}, \ref{D4.11} and note \ref{N2}, proof of (i)(b) and (ii)(b) follows from Definition \ref{D4.12} and proof of (i)(c) and (ii)(c) follows from Lemmas \ref{L4.4}, \ref{L4.5}, \ref{L4.8} and \ref{L4.9}.
\end{proof}
\begin{cor}
Let $\mathscr{A} = \mathbb{C}(x).$ A partition algebra of $\mathbb{P}_{2k}(x^2)$ is a tabular algebra  together with a table datum $(\Lambda, \Gamma, B, M[(r, (s_1, s_2))], C, \ast)$ with $f = id$ and $s_2=0.$
\end{cor}

\section{\textbf{A Cellular Basis of the algebra of $\mathbb{Z}_2$-relations and signed partition algebras}}

In this section, we compute a cellular basis for the algebra of $\mathbb{Z}_2$-relations and signed partition algebras by making use of the basis defined in Lemma \ref{L4.4} and also by using  cellular bases of the group
algebras $\mathscr{A}[\mathbb{Z}_2 \wr \mathfrak{S}_k]$ and
$\mathscr{A}[\mathfrak{S}_k]$ given  in \cite{RGA}.

\begin{defn} \label{D5.1}

Define,
\begin{enumerate}
  \item[(i)]$\Lambda' := \left\{ ((r, (s_1, s_2)), ((\lambda_1, \lambda_2), \mu)) \ | \ (r, (s_1, s_2)) \in \Lambda \right\}$
  \item[(ii)]$\widetilde{\Lambda}' := \left\{ ((r, (s_1, s_2)), ((\lambda_1, \lambda_2), \mu)) \ | \ (r, (s_1, s_2)) \in \widetilde{\Lambda} \right\}$
\end{enumerate}

 with the order given by

 \centerline{$\big( r', ( s'_1,
 s'_2) , ((\lambda'_1, \lambda'_2), \mu')\big) \geq \big( r,  (s_1,
s_2) , ((\lambda_1, \lambda_2), \mu)\big)$}

 if and only if
\begin{enumerate}
    \item[(a)] $r' \geq r$ or
    \item[(b)] $r' = r$ and  $(s'_1, s'_2) \geq (s_1, s_2)$ i.e., $s'_1 + s'_2 > s_1 + s_2$
    \item[(c)] $r' = r,  (s'_1,  s'_2) = (s_1,
    s_2)$ and $(\lambda'_1, \lambda'_2)  \rhd (\lambda_1, \lambda_2).$
\item[(d)] $r = r', (s'_1,  s'_2) = (s_1,
    s_2)$, $(\lambda'_1, \lambda'_2) = (\lambda_1, \lambda_2)$ and $\mu' \rhd \mu.$
\end{enumerate}
\end{defn}

\begin{defn} \label{D5.2}

Let $[\lambda], [\mu]$ denote the trivial representation of $\lambda, \mu.$

For $((r, (s_1, s_2)), ((\lambda_1, \lambda_2), \mu)) \in \Lambda'$ and $((r, (s_1, s_2)), ((\lambda_1, \lambda_2), \mu)) \in \widetilde{\Lambda}' ,$ define

\centerline{$M' \left[ (r, (s_1, s_2)), ((\lambda_1, \lambda_2), \mu)\right] := M[(r,(s_1, s_2))] \times M^{((\lambda_1, \lambda_2), \mu)}$ }

\centerline{$\widetilde{M}' \left[ (r, (s_1, s_2)), ((\lambda_1, \lambda_2), \mu)\right] := \widetilde{M}[(r,(s_1, s_2))] \times M^{((\lambda_1, \lambda_2), \mu)}$}

\NI where $M^{((\lambda_1, \lambda_2), \mu)} := \{ ((s_{\lambda_1}, s_{\lambda_2}), s_{\mu}) \ | \ s_{\lambda_1}, s_{\lambda_2} \text{ and } s_{\mu} \text{ are the standard tableaus of shape } \lambda_1, \lambda_2 \\ \text{ and } \mu  \text{ respectively}\}.$
\begin{enumerate}
  \item[(a)] if $s_1 \neq 0$ and $s_2 \neq 0$ then
\begin{enumerate}
  \item[(i)]  $M' \Big[\big( r, (s_1, s_2) , ((\lambda_1, \lambda_2),
\mu)\big) \Big] = \Big\{ \Big((d', P), ((t_{\lambda_1},
t_{\lambda_2}), t_{\mu}) \Big) \ \Big| \ (d', P) \in M[(r,(s_1, s_2))]
 ,  \\
   t_{\lambda_1}, t_{\lambda_2} \text{
and } t_{\mu} \text{ are the standard tableaux of shapes }
\lambda_1, \lambda_2 \text{ and }  \mu \text{ respectively}
\Big\},$
  \item[(ii)]  $\widetilde{M}' \Big[\big( r, (s_1, s_2) , ((\lambda_1, \lambda_2),
\mu)\big) \Big] = \Big\{ \Big((\widetilde{d}', \widetilde{P}), ((t_{\lambda_1},
t_{\lambda_2}), t_{\mu}) \Big) \ \Big| \ (\widetilde{d}', \widetilde{P}) \in M[(r,(s_1, s_2))]
 , \\   t_{\lambda_1}, t_{\lambda_2} \text{
and } t_{\mu} \text{ are the standard tableaux of shapes }
\lambda_1, \lambda_2 \text{ and }  \mu \text{ respectively}
\Big\},$
\end{enumerate}

\item[(b)] If $s_1 \neq 0$ and $s_2 = 0$ then
\begin{enumerate}
  \item[(i)] $M' \Big[\big( r,
(s_1, 0)  , ((\lambda_1, \lambda_2), \Phi)\big) \Big] = \Big\{
((d', P), (t_{\lambda_1}, t_{\lambda_2})) \ \Big| \ (d', P) \in M[(r,(s_1, 0))],   t_{\lambda_1}\\\text{
and } t_{\lambda_2} \text{ are  }   \text{  the standard
tableaux of shapes } \lambda_1 \text{ and } \lambda_2 \text{
respectively} \Big\},$
   \item[(ii)]  $\widetilde{M}' \Big[\big( r,
(s_1, 0)  , ((\lambda_1, \lambda_2), \Phi)\big) \Big] = \Big\{
((\widetilde{d}', \widetilde{P}), (t_{\lambda_1}, t_{\lambda_2})) \ \Big| \ (\widetilde{d}', \widetilde{P}) \in \widetilde{M}[(r,(s_1, 0))],   t_{\lambda_1}\\\text{
and } t_{\lambda_2} \text{ are  }   \text{  the standard
tableaux of shapes } \lambda_1 \text{ and } \lambda_2 \text{
respectively} \Big\},$
\end{enumerate}

\item[(c)] If $s_1 = 0$ and $s_2 \neq 0$ then

\begin{enumerate}
  \item[(i)] $M' \Big[\big(r,
(0, s_2)  ,  ((\Phi, \Phi), \mu)\big) \Big] = \Big\{ ((d', P),
t_{\mu}) \ \Big| \ (d', P) \in M[(r,(0, s_2))],
 t_{\mu}  \text{
is a standard} \\ \text{ tableau } \text{ of shape } \mu  \Big\},$
  \item[(ii)]$\widetilde{M}' \Big[\big(r,
(0, s_2)  ,  ((\Phi, \Phi), \mu)\big) \Big] = \Big\{ ((\widetilde{d}', \widetilde{P}),
t_{\mu}) \ \Big| \ (\widetilde{d}', \widetilde{P}) \in \widetilde{M}[(r,(0, s_2))],
 t_{\mu}  \text{
is a standard} \\ \text{ tableau } \text{ of shape } \mu  \Big\},$
\end{enumerate}

\item[(d)] If $r=0, s_1 = 0$ and $s_2=0$ then
\begin{enumerate}
  \item[(i)] $M' \Big[\big(0,
(0, 0)  ,  ((\Phi, \Phi), \Phi)\big) \Big] = \left\{ (d', \Phi)
 \ \Big| \ (d', \Phi) \in M[(0, (0, 0))]
\right\}$
  \item[(ii)] $\widetilde{M}' \Big[\big(0,
(0, 0)  ,  ((\Phi, \Phi), \Phi)\big) \Big] = \left\{ (\widetilde{d}', \Phi)
 \ \Big| \ (\widetilde{d}', \Phi) \in M[(0, (0, 0))]
\right\}$
\end{enumerate}

\NI where $s_1 = \natural \{ C : C \text{ is a connected component of } $P$ \text{ such that } H_C^P = \{e\}\}$ and $s_2 = \natural \{ C : C \text{ is a }\\ \text{connected component of } $P$ \text{ such that } H_C^P = \mathbb{Z}_2 \}.$
\end{enumerate}

\end{defn}

\begin{defn}\label{D5.3}
Let
\begin{enumerate}
  \item[(i)] $C' : \underset{(r, (s_1, s_2), ((\lambda_1, \lambda_2),\mu)) \in \Lambda'} {\coprod} M' \left[ (r, (s_1, s_2), ((\lambda_1, \lambda_2), \mu))\right] \times M' \left[ (r, (s_1, s_2), ((\lambda_1, \lambda_2), \mu))\right] \rightarrow A_k^{\mathbb{Z}_2}$

      \NI be defined as

      \centerline{$C' [ ((d', P), ((s_{\lambda_1}, s_{\lambda_2}), s_{\mu})), ((d'', Q), ((t_{\lambda_1}, t_{\lambda_2}), t_{\mu}))] = C^{m^{\lambda}_{s_{\lambda}, t_{\lambda}} m^{\mu}_{s_{\mu} t_{\mu}}}_{(d', P), (d'', Q)}$}

  \item[(ii)] $\widetilde{C}' : \underset{(r, (s_1, s_2), ((\lambda_1, \lambda_2),\mu)) \in \widetilde{\Lambda}'} {\coprod} \widetilde{M}' \left[ (r, (s_1, s_2), ((\lambda_1, \lambda_2), \mu))\right] \times \widetilde{M}' \left[ (r, (s_1, s_2), ((\lambda_1, \lambda_2), \mu))\right] \rightarrow \overrightarrow{A}_k^{\mathbb{Z}_2}$

      \NI be defined as

      \centerline{$\widetilde{C}' [ ((\widetilde{d}', \widetilde{P}), ((s_{\lambda_1}, s_{\lambda_2}), s_{\mu})), ((\widetilde{d}'', \widetilde{Q}), ((t_{\lambda_1}, t_{\lambda_2}), t_{\mu}))] = C^{m^{\lambda}_{s_{\lambda}, t_{\lambda}} m^{\mu}_{s_{\mu} t_{\mu}}}_{(\widetilde{d}', \widetilde{P}), (\widetilde{d}'', \widetilde{Q})}$}

\end{enumerate}

 \NI where $m^{\lambda}_{s_{\lambda}, t_{\lambda}}$ and $m^{\mu}_{s_{\mu} t_{\mu}}$ are cellular basis for the algebras $\mathscr{A} \left[ \mathbb{Z}_2 \wr \mathfrak{S}_{s_1}\right]$ and $\mathscr{A} \left[ \mathfrak{S}_{s_2}\right]$ respectively.

\end{defn}

\begin{thm}\label{T5.4}
Let $A_k^{\mathbb{Z}_2}\left(\overrightarrow{A}_k^{\mathbb{Z}_2} \right)$ be the $\mathscr{A}$-
algebra defined in Definition \ref{D2.5}(\ref{D2.7}).
\begin{enumerate}
  \item[(i)] The algebra of $\mathbb{Z}_2$ relations $\mathscr{A}[A_k^{\mathbb{Z}_2}]$ is a cellular algebra with a cell datum $(\Lambda',M', C', \ast')$ given as follows:
      \begin{enumerate}
    \item[(a)] $\Lambda'$ is a partially ordered set where $\Lambda'$ is
     as in Definition \ref{D5.1}.
    \item[(b)] $\ast$ is the unique anti involution of
    $A_k^{\mathbb{Z}_2}.$
    \item[(c)] \begin{enumerate}
                       \item[1.] $a C'^{m^{\lambda}_{s_{\lambda}, t_{\lambda}} m^{\mu}_{s_{\mu} t_{\mu}}}_{(d', P), (d'', Q)} \equiv \underset{S' \in M' \Big[ \big(r, (
s_1, s_2) , ((\lambda_1, \lambda_2), \mu)\big) \Big]}\sum r_a[(d''', P'''), (d', P)] C'^{ m^{\lambda}_{s'_{\lambda}, t_{\lambda}} m^{\mu}_{s'_{\mu}, t_{\mu}}}_{(d''', P'''), (d'', Q)}$

$\hspace{7cm} \text{ mod } A_k^{\mathbb{Z}_2}
\Big( < \big(r, (s_1, s_2), ((\lambda_1, \lambda_2), \mu)
\big)\Big),$

\NI where  $r_a[(d''', P'''), (d', P)]$ is independent of $(d'', Q).$

                 \item[2.] $a C'_{(d, \Phi), (d', \Phi)} \equiv \underset{(d'', \Phi) \in M' \big[ (0,(0, 0),((\Phi, \Phi))\big]}{\sum} r_a[(d'', \Phi), (d, \Phi)] C'_{(d'', \Phi), (d, \Phi)}.$
    \end{enumerate}
\end{enumerate}

  \item[(ii)]The signed partition algebra is a cellular algebra $\mathscr{A}[\overrightarrow{A}_k^{\mathbb{Z}_2}]$ with a cell datum $ (\widetilde{\Lambda}', \widetilde{M}', \widetilde{C}', \widetilde{\ast}')$ given as follows:
 \begin{enumerate}
    \item[(a)] $\widetilde{\Lambda}'$ is a partially ordered set where $\widetilde{\Lambda}'$ is
     as in Definition \ref{D5.1}.
    \item[(b)] $\widetilde{\ast}$ is the unique anti involution of
    $\overrightarrow{A}_k^{\mathbb{Z}_2}.$
    \item[(c)] \begin{enumerate}
                       \item[1.] $\widetilde{a} \widetilde{C'}^{m^{\lambda}_{s_{\lambda}, t_{\lambda}} m^{\mu}_{s_{\mu} t_{\mu}}}_{(\widetilde{d}', \widetilde{P}), (\widetilde{d}'', \widetilde{Q})} \equiv \underset{\widetilde{S}' \in \widetilde{M}' \Big[ \big(r, (
s_1, s_2) , ((\lambda_1, \lambda_2), \mu)\big) \Big]}\sum r_{\widetilde{a}}[(\widetilde{d}''', \widetilde{P}'''), (\widetilde{d}', \widetilde{P})] \widetilde{C'}^{ m^{\lambda}_{s'_{\lambda}, t_{\lambda}} m^{\mu}_{s'_{\mu}, t_{\mu}}}_{(\widetilde{d}''', \widetilde{P}'''), (\widetilde{d}'', \widetilde{Q})}$

$\hspace{7cm} \text{ mod } \overrightarrow{A}_k^{\mathbb{Z}_2}
\Big( < \big(r, (s_1, s_2), ((\lambda_1, \lambda_2), \mu)
\big)\Big),$

\NI where  $r_{\widetilde{a}}[(\widetilde{d}''', \widetilde{P}'''), (\widetilde{d}', \widetilde{P})]$ is independent of $(\widetilde{d}'',\widetilde{Q}).$

                 \item[2.] $\widetilde{a} \widetilde{C}'_{(\widetilde{d}, \Phi), (\widetilde{d}', \Phi)} \equiv \underset{(\widetilde{d}'', \Phi) \in \widetilde{M}' \big[ (0,(0, 0),((\Phi, \Phi))\big]}{\sum} r_{\widetilde{a}}[(\widetilde{d}'', \Phi), (\widetilde{d}, \Phi)] \widetilde{C}'_{(\widetilde{d}'', \Phi), (\widetilde{d}, \Phi)}.$
    \end{enumerate}
\end{enumerate}

\end{enumerate}

\end{thm}

\begin{proof}
The proof follows from Theorem 4.2.1 of \cite{GM2}, Lemma \ref{L4.8} and  Theorem \ref{T4.14}.
\end{proof}

\begin{rem}
From (1.8) of \cite{GL}, $A_k^{\mathbb{Z}_2}\left( \overrightarrow{A}_k^{\mathbb{Z}_2}\right)$ is a cellular algebra over any field $K$ with cell datum $(\Lambda', M', C', \ast') \\ ((\widetilde{\Lambda}', \widetilde{M}', \widetilde{C}', \widetilde{\ast}'))$ where $(\Lambda', M', C', \ast')((\widetilde{\Lambda}', \widetilde{M}', \widetilde{C}', \widetilde{\ast}'))$ is as in Theorem \ref{T5.4}.
\end{rem}
\begin{cor}
Let $\mathbb{P}_{2k}(x^2)$ be the $\mathscr{A}$-
algebra defined in Definition \ref{D2.8}. Then
$\mathbb{P}_{2k}(x^2)$ has a cell datum $(\Lambda',
M', C', \ast')$ with $f = id$ and $s_2=0.$
\end{cor}

\section{\textbf{Modular Representations of the algebra of $\mathbb{Z}_2$ relations and signed partition algebras}}

In this section, we give a description of the complete set of irreducible modules for the algebra of $\mathbb{Z}_2$ relations $A_k^{\mathbb{Z}_2}$ and signed partition algebras $\overrightarrow{A}_k^{\mathbb{Z}_2}$over any field.

\begin{defn} \label{D6.1}
Let $r = 2s_1 + s_2.$ For $0 \leq r \leq 2k$ and $((r, (s_1, s_2)), ((\lambda_1, \lambda_2), \mu)) \in \Lambda' \\ \left(((r, (s_1, s_2)), ((\lambda_1, \lambda_2), \mu)) \in \widetilde{\Lambda}' \right),$ put $\lambda = (\lambda_1, \lambda_2).$

The left cell module $W \left[(r, (s_1, s_2)), ((\lambda_1, \lambda_2), \mu) \right]\left(\widetilde{W} \left[(r, (s_1, s_2)), ((\lambda_1, \lambda_2), \mu) \right] \right)$ for the cellular algebra $\mathscr{A} \left[ A_k^{\mathbb{Z}_2}\right]\left(\mathscr{A} \left[ \overrightarrow{A}_k^{\mathbb{Z}_2}\right] \right)$ is defined as follows:

\begin{enumerate}
  \item[(i)] $W \left[(r, (s_1, s_2)), ((\lambda_1, \lambda_2), \mu) \right]$ is a free $\mathscr{A}$-module with basis

\centerline{$\left\{C_{S}^{C_{s_1, s_2}(s)} = C'^{m^{\lambda}_{s_{\lambda}}m^{\mu}_{s_{\mu}}}_{S} \ \Big| \ S = (d, P), s =  ((s_{\lambda_1}, s_{\lambda_2}), s_{\mu}) \in M' \left[(r, (s_1, s_2)), ((\lambda_1, \lambda_2), \mu) \right]\right\}$}
and $A_k^{\mathbb{Z}_2}$-action is defined on the basis element  by $a$

$a C'^{m^{\lambda}_{s_{\lambda}} m^{\mu}_{s_{\mu}}}_{S} \equiv \underset{ (S', s') \in M' \Big[ \big(r, (
s_1, s_2) , ((\lambda_1, \lambda_2), \mu)\big) \Big]}\sum  C^{r_a(S', S) m^{\lambda}_{s'_{\lambda}} m^{\mu}_{s'_{\mu}}}_{S'}  \ \ \text{ mod } A_k^{\mathbb{Z}_2}
\Big( < \big(r, (s_1, s_2), ((\lambda_1, \lambda_2), \mu)
\big)\Big),$

\NI where $(S, s) = ((d, P), ((s_{\lambda_1}, s_{\lambda_2}), s_{\mu})), (S', s') = ((d', P'), ((s'_{\lambda_1}, s'_{\lambda_2}), s'_{\mu}))$  $r_a(S', S)$ is as in 3(a)(i) and (b)(i) of Theorem \ref{T5.4}.

  \item[(ii)] \NI $W \left[(r, (s_1, s_2)), ((\lambda_1, \lambda_2), \mu) \right]$ is a free $\mathscr{A}$-module with basis

\centerline{$\left\{\widetilde{C}_{\widetilde{S}}^{\widetilde{C}_{s_1, s_2}(s)} = \widetilde{C}'^{m^{\lambda}_{s_{\lambda}} m^{\mu}_{s_{\mu}}}_{\widetilde{S}} \ \Big| \ \widetilde{S} = (\widetilde{d}, \widetilde{P}), s =  ((s_{\lambda_1}, s_{\lambda_2}), s_{\mu}) \in \widetilde{M}' \left[(r, (s_1, s_2)), ((\lambda_1, \lambda_2), \mu) \right]\right\}$}
and $\overrightarrow{A}_k^{\mathbb{Z}_2}$-action is defined on the basis element  by $\widetilde{a}$

$\widetilde{a} \widetilde{C}'^{m^{\lambda}_{s_{\lambda}}m^{\mu}_{s_{\mu}}}_{\widetilde{S}} \equiv \underset{ (\widetilde{S}', s') \in \widetilde{M}' \Big[ \big(r, (
s_1, s_2) , ((\lambda_1, \lambda_2), \mu)\big) \Big]}\sum  \widetilde{C}^{r_{\widetilde{a}}(\widetilde{S}', \widetilde{S}) m^{\lambda}_{s'_{\lambda}} m^{\mu}_{s'_{\mu}}}_{\widetilde{S}'}  \ \ \text{ mod } \overrightarrow{A}_k^{\mathbb{Z}_2}
\Big( < \big(r, (s_1, s_2), ((\lambda_1, \lambda_2), \mu)
\big)\Big),$

\NI where $(\widetilde{S}, s) = ((\widetilde{d}, \widetilde{P}), ((s_{\lambda_1}, s_{\lambda_2}), s_{\mu})), (\widetilde{S}', s') = ((\widetilde{d}', \widetilde{P}'), ((s'_{\lambda_1}, s'_{\lambda_2}), s'_{\mu}))$  $r_a(\widetilde{S}', \widetilde{S})$ is as in 3(a)(ii) and (b)(ii) of Theorem \ref{T5.4}.

\end{enumerate}

\end{defn}
\begin{lem}\label{L6.2}
\begin{itemize}
  \item[(i)]  $C'^{m^{\lambda}_{s_{\lambda}, s_{\lambda}} m^{\mu}_{s_{\mu}, s_{\mu}}}_{S, S} \ C'^{m^{\lambda}_{t_{\lambda}, t_{\lambda}} m^{\mu}_{t_{\mu}, t_{\mu}}}_{T, T} \ \equiv \ \Phi_1((S, s), (T, t)) \ C'^{m^{\lambda}_{s_{\lambda}, t_{\lambda}} m^{\mu}_{s_{\mu}, t_{\mu}}}_{S, T} \ $

       $\hspace{10cm} \text{ mod } \left[ \tiny{ A_k^{\mathbb{Z}_2} <(r, (s_1, s_2), ((\lambda_1, \lambda_2), \mu)}\right]$

     \NI where

       $\begin{array}{llll}
      \Phi_1((S, s),  (T, t)) & = & x^{l(P \vee P')} \phi^{\lambda}_{\delta_1}(s_{\lambda}, t_{\lambda}) \ \phi^{\mu}_{\delta_2}(s_{\mu}, t_{\mu})& \text{ when conditions (a) and (b)}  \\
      & & & \text{ of Definition } \ref{D4.7} \text{ are satisfied}\\
      & = &  0 & \text{Otherwise}
      \end{array}$
  \item[(ii)]  $\widetilde{C}'^{m^{\lambda}_{s_{\lambda}, s_{\lambda}} m^{\mu}_{s_{\mu}, s_{\mu}}}_{\widetilde{S}, \widetilde{S}} \ \widetilde{C}'^{m^{\lambda}_{t_{\lambda}, t_{\lambda}} m^{\mu}_{t_{\mu}, t_{\mu}}}_{\widetilde{T}, \widetilde{T}} \ \equiv \ \Phi_1((\widetilde{S}, s), (\widetilde{T}, t)) \ \widetilde{C}'^{m^{\lambda}_{s_{\lambda}, t_{\lambda}} m^{\mu}_{s_{\mu}, t_{\mu}}}_{\widetilde{S}, \widetilde{T}} \ \ \ \ \ \text{ mod } \left[ \overrightarrow{A}_k^{\mathbb{Z}_2} <(r, (s_1, s_2), ((\lambda_1, \lambda_2), \mu)\right]$
     \NI where

       $\begin{array}{llll}
      \Phi_1((\widetilde{S}, s),  (\widetilde{T}, t)) & = & x^{l(\widetilde{P} \vee \widetilde{P}')} \phi^{\lambda}_{\delta_1}(s_{\lambda}, t_{\lambda}) \ \phi^{\mu}_{\delta_2}(s_{\mu}, t_{\mu})& \text{ when conditions (a) and (b)}  \\
      & & & \text{ of Definition } \ref{D4.7} \text{ are satisfied}\\
      & = &  0 & \text{Otherwise}
      \end{array}$

\end{itemize}

        \NI  $(S, s) = ((d, P), ((s_{\lambda_1}, s_{\lambda_2}), s_{\mu})), (\widetilde{S}, s) = ((\widetilde{d}, \widetilde{P}), ((s_{\lambda_1}, s_{\lambda_2}), s_{\mu})), (T, t) = ((d', P'), ((t_{\lambda_1}, t_{\lambda_2}), t_{\mu})), (\widetilde{T}, t) = ((\widetilde{d}', \widetilde{P}'), ((t_{\lambda_1}, t_{\lambda_2}), t_{\mu})), l(P \vee P')\left( l(\widetilde{P} \vee \widetilde{P}')\right) $ denotes the number of

          \NI connected components in $ d'.d''\left( \widetilde{d}' . \widetilde{d}''\right)$  excluding the union of all the connected components of $ P \text{ and } P' \left(  \widetilde{P} \text{ and } \widetilde{P}' \right)$,

          \NI $ m^{\lambda}_{s_{\lambda}, s_{\lambda}} \delta_{1} m^{\lambda}_{t_{\lambda}, t_{\lambda}} \equiv \phi^{\lambda}_{\delta_1}(s_{\lambda}, t_{\lambda}) m^{\lambda}_{s'_{\lambda}, t_{\lambda}}  \text{mod } \mathscr{H} \left(< (\lambda_1, \lambda_2)\right), m^{\mu}_{s_{\mu}, s_{\mu}} \delta_2 m^{\mu}_{t_{\mu}, t_{\mu}} \equiv \phi^{\mu}_{\delta_2}(s_{\mu}, t_{\mu})m^{\mu}_{s'_{\mu}, t_{\mu}}  \text{mod } \mathscr{H}' \left( < \mu \right)$

         \NI $\text{ as in Lemma 1.7 \cite{GL}}.$
\end{lem}
\begin{proof}
\NI \textbf{Proof of (i):}
Consider the product
\begin{eqnarray*}
   C'^{m^{\lambda}_{s_{\lambda}, s_{\lambda}} m^{\mu}_{s_{\mu}, s_{\mu}}}_{S, S} \ C'^{m^{\lambda}_{t_{\lambda}, t_{\lambda}} m^{\mu}_{t_{\mu}, t_{\mu}}}_{T, T}  &=&  x^{l(P \vee P')} C'^{m^{\lambda}_{s_{\lambda}, s_{\lambda}} m^{\mu}_{s_{\mu}, s_{\mu}} (\delta_1, \delta_2) m^{\lambda}_{t_{\lambda}, t_{\lambda}} m^{\mu}_{t_{\mu}, t_{\mu}}}_{S, T}
\end{eqnarray*}

where $\phi^r_{s_1, s_2}((d, P), (d', P')) = x^{l(P \vee P')} (\delta_1, \delta_2)$ is as in Definition \ref{D4.7},

We know that,

\begin{eqnarray}
 m^{\lambda}_{s_{\lambda}, s_{\lambda}} m^{\mu}_{s_{\mu}, s_{\mu}} (\delta_1, \delta_2) m^{\lambda}_{t_{\lambda}, t_{\lambda}} m^{\mu}_{t_{\mu}, t_{\mu}}  &=& m^{\lambda}_{s_{\lambda}, s_{\lambda}} \delta_1 m^{\lambda}_{t_{\lambda}, t_{\lambda}} m^{\mu}_{s_{\mu}, s_{\mu}} \delta_2 m^{\mu}_{t_{\mu}, t_{\mu}} \\\nonumber
   &=& \phi^{\lambda}_{\delta_1}(s_{\lambda}, t_{\lambda}) m^{\lambda}_{s'_{\lambda}, t_{\lambda}} \phi^{\mu}_{\delta_2}(s_{\mu}, t_{\mu}) m^{\mu}_{s'_{\mu}, t_{\mu}} \\\nonumber
   &=& \phi^{\lambda}_{\delta_1}(s_{\lambda}, t_{\lambda})\ \phi^{\mu}_{\delta_2}(s_{\mu}, t_{\mu}) \ m^{\lambda}_{s'_{\lambda}, t_{\lambda}} m^{\mu}_{s'_{\mu}, t_{\mu}}
   \end{eqnarray}
   where \NI $m^{\lambda}_{s_{\lambda}, s_{\lambda}} \delta_{1} m^{\lambda}_{t_{\lambda}, t_{\lambda}} \equiv \phi^{\lambda}_{\delta_1}(s_{\lambda}, t_{\lambda}) m^{\lambda}_{s'_{\lambda}, t_{\lambda}} \ \ \text{mod } \mathscr{H} \left(< (\lambda_1, \lambda_2)\right),$

\NI $m^{\mu}_{s_{\mu}, s_{\mu}} \delta_2 m^{\mu}_{t_{\mu}, t_{\mu}} \equiv \phi^{\mu}_{\delta_2}(s_{\mu}, t_{\mu})m^{\mu}_{s'_{\mu}, t_{\mu}} \ \ \text{mod } \mathscr{H}' \left( < \mu \right). $

   Substitute the above in the product $C'^{m^{\lambda}_{s_{\lambda}, s_{\lambda}} m^{\mu}_{s_{\mu}, s_{\mu}}}_{S, S} \ C'^{m^{\lambda}_{t_{\lambda}, t_{\lambda}} m^{\mu}_{t_{\mu}, t_{\mu}}}_{T, T}$ we get,
   \begin{eqnarray*}
     C'^{m^{\lambda}_{s_{\lambda}, s_{\lambda}} m^{\mu}_{s_{\mu}, s_{\mu}}}_{S, S} \ C'^{m^{\lambda}_{t_{\lambda}, t_{\lambda}} m^{\mu}_{t_{\mu}, t_{\mu}}}_{T, T} &=& x^{l(P \vee P')} \phi^{\lambda}_{\delta_1}(s_{\lambda}, t_{\lambda})\ \phi^{\mu}_{\delta_2}(s_{\mu}, t_{\mu}) C'^{m^{\lambda}_{s'_{\lambda}, t_{\lambda}} m^{\mu}_{s'_{\mu}, t_{\mu}} }_{S, T} \\
      &=& \Phi_1((S, s), (T, t)) \ C^{m^{\lambda}_{s'_{\lambda}, t_{\lambda}} m^{\mu}_{s'_{\mu}, t_{\mu}}}_{S, T}
        \end{eqnarray*}
        where $\Phi_1((S, s), (T, t)) = x^{l(P \vee P')} \phi^{\lambda}_{\delta_1}(s_{\lambda}, t_{\lambda})\ \phi^{\mu}_{\delta_2}(s_{\mu}, t_{\mu}).$

\textbf{Proof of (ii):} Proof of (ii) is same as  proof of (i).
\end{proof}

\begin{defn}\label{D6.3}
For $\big(r, (s_1, s_2), ((\lambda_1, \lambda_2), \mu)\big) \in
\Lambda'\left(\big(r, (s_1, s_2), ((\lambda_1, \lambda_2), \mu)\big) \in
\widetilde{\Lambda}' \right),$  the bilinear map $\phi_{s_1, s_2}^{\lambda, \mu} \\ \left( \widetilde{\phi}_{s_1, s_2}^{\lambda, \mu}\right)$ is
defined as
\begin{enumerate}
  \item[(i)] $ \phi_{s_1, s_2}^{\lambda, \mu} \big(  C'^{m^{\lambda}_{s_{\lambda}, s_{\lambda}} m^{\mu}_{s_{\mu}, s_{\mu}}}_{(d, P)},   C'^{m^{\lambda}_{t_{\lambda}, t_{\lambda}} m^{\mu}_{t_{\mu}, t_{\mu}}}_{(d', P')} \big) = \Phi_1((S, s), (T, t)), \ \ (S, s), (T, t) \in
M' \big[r, (s_1, s_2), ((\lambda_1, \lambda_2), \mu) \big] $
  \item[(ii)] $ \widetilde{\phi}_{s_1, s_2}^{\lambda, \mu} \big(  \widetilde{C}'^{m^{\lambda}_{s_{\lambda}, s_{\lambda}} m^{\mu}_{s_{\mu}, s_{\mu}}}_{(\widetilde{d}, \widetilde{P})},   \widetilde{C}'^{m^{\lambda}_{t_{\lambda}, t_{\lambda}} m^{\mu}_{t_{\mu}, t_{\mu}}}_{(\widetilde{d}', \widetilde{P}')} \big) = \Phi_1((\widetilde{S}, s), (\widetilde{T}, t)), \ \ (\widetilde{S}, s), (\widetilde{T}, t) \in
\widetilde{M}' \big[r, (s_1, s_2), ((\lambda_1, \lambda_2), \mu) \big] $
\end{enumerate}

\NI where $\Phi_1((S, s), (T, t))\left(\widetilde{\Phi}_1((\widetilde{S}, s), (\widetilde{T}, t)) \right)$ is as in Lemma \ref{L6.2}.

Put
\begin{enumerate}
  \item[(i)] $G^{\lambda, \mu}_{2s_1+s_2} = \left(\Phi_1((S, s), (T, t)) \right)_{(S, s), (T, t) \in M'\big[r, (s_1, s_2), ((\lambda_1, \lambda_2), \mu) \big]}$

      where

 $\begin{array}{llll}
      \Phi_1((S, s), (T, t)) & = & x^{l(P_i \vee P_j)} \phi^{\lambda}_{\delta_1}(s_{\lambda}, t_{\lambda}) \ \phi^{\mu}_{\delta_2}(s_{\mu}, t_{\mu})& \text{ when conditions (a) and (b)}  \\
      & & & \text{ of Definition } \ref{D4.7} \text{ are satisfied}\\
       & = &  0 & \text{Otherwise}
      \end{array}$

\NI where $(S, s) = ((d_i, P_i), ((s_{\lambda_1}, s_{\lambda_2}), s_{\mu})), (T, t) = ((d_j, P_j), ((t_{\lambda_1}, t_{\lambda_2}), t_{\mu}))$

  \item[(ii)]$\widetilde{G}^{\lambda, \mu}_{2s_1+s_2} = \left(\widetilde{\Phi}_1((\widetilde{S}, s), (\widetilde{T}, t)) \right)_{(\widetilde{S}, s), (\widetilde{T}, t) \in \widetilde{M}'\big[r, (s_1, s_2), ((\lambda_1, \lambda_2), \mu) \big]}$

      where

 $\begin{array}{llll}
      \widetilde{\Phi}_1((\widetilde{S}, s), (\widetilde{T}, t)) & = & x^{l(\widetilde{P}_i \vee \widetilde{P}_j)} \phi^{\lambda}_{\delta_1}(s_{\lambda}, t_{\lambda}) \ \phi^{\mu}_{\delta_2}(s_{\mu}, t_{\mu})& \text{ when conditions (a) and (b)}  \\
      & & & \text{ of Definition } \ref{D4.7} \text{ are satisfied}\\
       & = &  0 & \text{Otherwise}
      \end{array}$

\NI where $(\widetilde{S}, s) = ((\widetilde{d}_i, \widetilde{P}_i), ((s_{\lambda_1}, s_{\lambda_2}), s_{\mu})), (\widetilde{T}, t) = ((\widetilde{d}_j, \widetilde{P}_j), ((t_{\lambda_1}, t_{\lambda_2}), t_{\mu})),$
\end{enumerate}

        \NI $l(P_i \vee P_j)\left( l(\widetilde{P}_i \vee \widetilde{P}_j)\right) $ denotes the number of  connected components in $ d'.d''\left(\widetilde{d}' . \widetilde{d}'' \right)$  excluding the union of all the connected components of $ P_i \text{ and } P_j\left( \widetilde{P}_i \text{ and } \widetilde{P}_j\right) $,

          \NI $ m^{\lambda}_{s_{\lambda}, s_{\lambda}} \delta_{1} m^{\lambda}_{t_{\lambda}, t_{\lambda}} \equiv \phi^{\lambda}_{\delta_1}(s_{\lambda}, t_{\lambda}) m^{\lambda}_{s'_{\lambda}, t_{\lambda}}  \text{mod } \mathscr{H} \left(< (\lambda_1, \lambda_2)\right), m^{\mu}_{s_{\mu}, s_{\mu}} \delta_2 m^{\mu}_{t_{\mu}, t_{\mu}} \equiv \phi^{\mu}_{\delta_2}(s_{\mu}, t_{\mu})m^{\mu}_{s'_{\mu}, t_{\mu}}  \text{mod } \mathscr{H}' \left( < \mu \right)$

         \NI $\text{ as in Lemma 1.7 \cite{GL}}.$

    \NI  $G^{\lambda, \mu}_{2s_1+s_2}\left( \widetilde{G}^{\lambda, \mu}_{2s_1+s_2}\right)$ is called the \textbf{Gram matrix of the cell module} $W \left[(r, (s_1, s_2)), ((\lambda_1, \lambda_2), \mu)\right] \\ \left( \widetilde{W} \left[(r, (s_1, s_2)), ((\lambda_1, \lambda_2), \mu)\right] \right).$

\end{defn}

\begin{defn}\label{D6.4}
For $\big(r, (s_1, s_2), ((\lambda_1, \lambda_2), \mu) \big) \in
\Lambda'\left( \big(r, (s_1, s_2), ((\lambda_1, \lambda_2), \mu) \big) \in
\widetilde{\Lambda}'\right),$ define

\begin{enumerate}
  \item[(i)] $Rad \big(W \big[r, (s_1, s_2), ((\lambda_1, \lambda_2), \mu)
\big]\big) = \big\{ x \in W \big[r, (s_1, s_2), ((\lambda_1,
\lambda_2), \mu) \big] \ | \ $

$ \hspace{5cm} \phi_{s_1, s_2}^{\lambda, \mu} (x, y) =0 \ \ \forall \
y \in W \big[r, (s_1, s_2), ((\lambda_1, \lambda_2), \mu) \big]
\big\},$
  \item[(ii)] $Rad \big(\widetilde{W} \big[r, (s_1, s_2), ((\lambda_1, \lambda_2), \mu)
\big]\big) = \big\{ x \in \widetilde{W} \big[r, (s_1, s_2), ((\lambda_1,
\lambda_2), \mu) \big] \ | \ $

$ \hspace{5cm} \widetilde{\phi}_{s_1, s_2}^{\lambda, \mu} (x, y) =0 \ \ \forall \
y \in \widetilde{W} \big[r, (s_1, s_2), ((\lambda_1, \lambda_2), \mu) \big]
\big\},$
\end{enumerate}

\NI where $(S, s) = ((d, P), ((s_{\lambda_1}, s_{\lambda_2}), s_{\mu})), (\widetilde{S}, s) = ((\widetilde{d}, \widetilde{P}), ((s_{\lambda_1}, s_{\lambda_2}), s_{\mu})), (T, t) = ((d', P'), ((t_{\lambda_1}, t_{\lambda_2}), t_{\mu}))$ and $(\widetilde{T}, t) = ((\widetilde{d}', \widetilde{P}'), ((t_{\lambda_1}, t_{\lambda_2}), t_{\mu})).$
\end{defn}
\begin{notn}\label{N6.5}

Let \begin{enumerate}
      \item[(i)] $\Lambda'_0  = \{(r, (s_1, s_2), ((\lambda_1, \lambda_2), \mu))
\in \Lambda' \ | \ \phi^{\lambda, \mu}_{s_1, s_2} \neq 0 \}.$
      \item[(ii)]$\widetilde{\Lambda}'_0  = \{(r, (s_1, s_2), ((\lambda_1, \lambda_2), \mu))
\in \widetilde{\Lambda}' \ | \ \widetilde{\phi}^{\lambda, \mu}_{s_1, s_2} \neq 0 \}.$
    \end{enumerate}
\end{notn}

\begin{thm}\label{T6.6}
Let $\mathbb{K}(x)$ be a field.
  For $(r, (s_1, s_2), ((\lambda_1, \lambda_2), \mu)) \in \Lambda'_0\left( (r, (s_1, s_2), ((\lambda_1, \lambda_2), \mu)) \in \widetilde{\Lambda}'_0\right), $

  let \begin{enumerate}
        \item[(i)]  $D^{(r, (s_1, s_2), ((\lambda_1, \lambda_2), \mu)} =\ds \frac{W
\big[ r, (s_1, s_2), ((\lambda_1, \lambda_2), \mu)\big]}{\text{Rad
} \big(W \big[ r, (s_1, s_2), ((\lambda_1, \lambda_2),
\mu)\big]\big)}$,
        \item[(ii)]  $\widetilde{D}^{(r, (s_1, s_2), ((\lambda_1, \lambda_2), \mu)} =\ds \frac{\widetilde{W}
\big[ r, (s_1, s_2), ((\lambda_1, \lambda_2), \mu)\big]}{\text{Rad
} \big(\widetilde{W} \big[ r, (s_1, s_2), ((\lambda_1, \lambda_2),
\mu)\big]\big)}$.

      \end{enumerate}

\begin{enumerate}
  \item [(a)]  $D^{(r, (s_1, s_2), ((\lambda_1, \lambda_2), \mu)} \neq 0 \left( \widetilde{D}^{(r, (s_1, s_2), ((\lambda_1, \lambda_2), \mu)} \neq 0\right)$  if and only if  $\lambda = (\lambda_1, \lambda_2)$ is $p$- restricted and $\mu$ is $p$- restricted and it is absolutely irreducible over a field of characteristic p.
      \item [$(a)'$]  $D^{(r, (s_1, s_2), ((\lambda_1, \lambda_2), \mu)} \neq 0 \left( \widetilde{D}^{(r, (s_1, s_2), ((\lambda_1, \lambda_2), \mu)} \neq 0\right)$   and it is  absolutely irreducible over a field of \\ characteristic 0.
  \item [(b)] $D^{(r, (s_1, 0), (\lambda_1, \lambda_2)} \neq 0 \left( \widetilde{D}^{(r, (s_1, 0), (\lambda_1, \lambda_2)} \neq 0\right)$ if and only if  $\lambda = (\lambda_1, \lambda_2)$ is $p$- restricted  and it is \\ absolutely  irreducible over a  field of characteristic p.
      \item [$(b)'$] $D^{(r, (s_1, 0), (\lambda_1, \lambda_2)} \neq 0 \left( \widetilde{D}^{(r, (s_1, 0), (\lambda_1, \lambda_2)} \neq 0\right)$  and it is absolutely  irreducible over a  field of \\ characteristic 0.
  \item[(c)] $D^{(r, (0, s_2),  \mu)} \neq 0 \left( \widetilde{D}^{(r, (0, s_2),  \mu)} \neq 0\right)$  if and only if  $\mu$ is $p$- restricted and it is absolutely irreducible over a  field of characteristic p.
  \item[$(c)'$] $D^{(r, (0, s_2),  \mu)} \neq 0 \left( \widetilde{D}^{(r, (0, s_2),  \mu)} \neq 0\right)$   and it is absolutely irreducible over a  field of characteristic 0.
  \item[(d)] $D^{(0, \Phi)}\left( \widetilde{D}^{(0, \Phi)}\right) $ is non-zero and it is absolutely irreducible over a field of characteristic 0.
  \end{enumerate}

 \end{thm}
 \begin{proof}

 We shall show that $\Phi_1((S, s), (T, t)) \neq 0$ for some $(S, s),  (T, t).$

 Consider $(S, s) = ((d, P), ((s_{\lambda_1}, s_{\lambda_2}), s_{\mu}))$ and $(S', s') = ((d, P), ((s'_{\lambda_1}, s'_{\lambda_2}), s'_{\mu}))$ then
 $$\Phi_1((S, s), (S', s')) = x^{l(P \vee P)} \phi_1(s_{\lambda}, s'_{\lambda}) \phi_1(s_{\mu}, s'_{\mu}),$$
where $ \lambda = (\lambda_1, \lambda_2), \phi_1(s_{\lambda}, s'_{\lambda})$ and $ \phi_1(s_{\mu}, s'_{\mu})$ are the bilinear forms of the cell module $W^{\lambda}$ and $W^{\mu}$ of the cellular algebras $k[\mathbb{Z}_2 \wr \mathfrak{S}_{s_1}]$ and $K[\mathfrak{S}_{s_2}]$ respectively.

 We know that $\phi_1(s_{\lambda}, s'_{\lambda}) \neq 0$ and $ \phi_1(s_{\mu}, s_{\mu}) \neq 0$ for some $s'_{\lambda}$ and $s'_{\mu}$ which implies that

 \centerline{$\Phi_1((S, s), (T, t)) \neq 0 $}

for some $(S, s) = \big((d, P), ((s_{\lambda_1}, s_{\lambda_2}), s_{\mu})\big), (T, t) = \big((d, Q), ((t_{\lambda_1}, t_{\lambda_2}), t_{\mu})\big).$

 Conversely, assume that $\Phi_1((S, s), (T, t)) \neq 0 \ \ \ \ \text{ for some } (S, s), (T, t).$
 $$i.e., \Phi_1((S, s), (T, t)) = x^{l(P \vee Q)} \phi^{\lambda}_{\delta_1}(s_{\lambda}, t_{\lambda}) \phi^{\mu}_{\delta_2}(s_{\mu}, t_{\mu}) \neq 0$$ which implies that
 \begin{equation}\label{E6.1}
 \phi^{\lambda}_{\delta_1}(s_{\lambda}, t_{\lambda}) \neq 0,  \phi^{\mu}_{\delta_2}(s_{\mu}, t_{\mu}) \neq 0
 \end{equation}
where $\phi^r_{s_1, s_2}((d, P), (d', Q)) = x^{l(P \vee Q)} (\delta_1, \delta_2)$ is as in Definition \ref{D4.7},

\NI $m^{\lambda}_{s_{\lambda}, s_{\lambda}} \delta_{1} m^{\lambda}_{t_{\lambda}, t_{\lambda}} \equiv \phi^{\lambda}_{\delta_1}(s_{\lambda}, t_{\lambda}) m^{\lambda}_{s_{\lambda}, t_{\lambda}} \ \ \text{mod } \mathscr{H} \left(< (\lambda_1, \lambda_2)\right)$ and

\NI $m^{\mu}_{s_{\mu}, s_{\mu}} \delta_2 m^{\mu}_{t_{\mu}, t_{\mu}} \equiv \phi^{\mu}_{\delta_2}(s_{\mu}, t_{\mu})m^{\mu}_{s_{\mu}, t_{\mu}} \ \ \text{mod } \mathscr{H}' \left( < \mu \right).$

\NI Also we know that by proof of (ii) of proposition 2.4 in \cite{GL},

 $$\phi^{\lambda}_{\delta_1}(s_{\lambda}, t_{\lambda}) = \sum r_{\delta^{\lambda}_1}(s'_{\lambda}, t_{\lambda}) \phi_1(s_{\lambda}, t_{\lambda}) \text{ and } \phi^{\mu}_{\delta_2}(s_{\mu}, t_{\mu}) = \sum r_{\delta^{\mu}_2}(s'_{\mu}, t_{\mu}) \phi_1(s_{\mu}, t_{\mu})$$

 By equation (\ref{E6.1}) we have,

 $\phi_1(s_{\lambda}, t_{\lambda}) \neq 0$ and $\phi_1(s_{\mu}, t_{\mu}) \neq 0$ for some $t_{\lambda}$ and $t_{\mu}.$

 Thus the proof of (a), (b), (c) follows from \cite{RGA} and (7.6) of \cite{RG} and the absolute irreducibility follows Proposition 3.2 of \cite{GL}.
 \end{proof}

\end{document}